\documentclass[10pt,reqno]{amsart}
\usepackage{fullpage}
\usepackage{amsthm,amsmath,amsfonts,amssymb,euscript,hyperref,graphics,color,slashed,mathrsfs,tikz,wrapfig,float,caption,enumerate,graphicx,mathtools,extarrows,cite,bm}
\allowdisplaybreaks

\hypersetup{colorlinks, citecolor=blue, filecolor=blue, linkcolor=blue, urlcolor=blue}

\numberwithin{equation}{section}

\newtheorem{Theorem}{Theorem}[section]
\newtheorem{Lemma}[Theorem]{Lemma}

\newtheorem{Corollary}[Theorem]{Corollary}
\newtheorem{Definition}[Theorem]{Definition}
\newtheorem{Remark}[Theorem]{Remark}

\title[Existence and nonexistence]{Existence and nonexistence of viscosity solutions for a class of degenerate/singular eigenvalue type equations}

\author[Mengni Li]{}
\address{Southeast University}
\email{krisymengni@gmail.com}
\author[You Li]{}
\address{Xiangtan University}
\email{thu3141@126.com}

\thanks{This work is supported in part by National Natural Science Foundation of China (Nos. 12301263, 12201107 and 12141103) and Natural Science Foundation of Beijing Municipality (No. 1212002). }

\begin{document}

\maketitle 

\centerline{Mengni Li and You Li\footnote{Corresponding author: You Li}}

\begin{abstract}
	
This paper is devoted to a complete classification on the existence and nonexistence results of viscosity solutions to the general Dirichlet problem for a class of  eigenvalue type equations. With the distance function included in the right-hand side, this type of equations 
can be degenerate and (or) singular near the boundary of uniformly convex domains. One highlight is that all cases related to the exponent of the distance function are investigated. Moreover, when viscosity solutions exist, we derive a series of global estimates based on the distance function.  The key ingredients of this paper include adaptions  of the Perron method and comparison principle as well as  constructions of suitable classical sub-solutions and  super-solutions. 

\noindent{\bf \\ Running head: } Existence and nonexistence
\noindent{\bf \\ Key words and phrases: } Dirichlet problem, nonzero boundary condition, existence, nonexistence, global $\log$-type estimate
\noindent{\bf \\ 2020 Mathematics Subject Classification: } 35D40, 35A01,   35J70, 35J75, 35P15 
\end{abstract}
		
	 \tableofcontents
	  
\section{Introduction}\label{section: introduction}

In this paper, we study the existence and nonexistence of viscosity solutions to the Dirichlet problem: 
\begin{align*}
F\left(\lambda_1(D^2 u),\cdots,\lambda_n(D^2 u)\right)&=g(x)(d(x))^{\alpha}\text{\hspace{0.2cm} in } \Omega,\stepcounter{equation}\tag{\theequation}\label{equation}\\
u&=\varphi\text{\hspace{1.7cm} on } \partial\Omega,\stepcounter{equation}\tag{\theequation}\label{bcondition}
\end{align*}
where $\Omega\subseteq \mathbb{R}^n(n \geq 2)$ refers to a convex domain having non-empty boundary $\partial\Omega$; $u:\Omega\to\mathbb{R}$ is a convex function; $\lambda_1(D^2u),\cdots,\lambda_n(D^2u)$ represent all the eigenvalues of the Hessian 
$D^2u$; $F$ denotes a given operator; $g$ and $\varphi$ are given functions; and $d(x)$ indicates the distance from the point $x$ to the boundary $\partial\Omega$:  $$d(x)=\operatorname{dist}(x,\partial\Omega).$$

\subsection{Assumptions, eigenvalue analysis, and examples}

The following four sets of assumptions are made regarding the domain $\Omega$,  operator $F$,  right-hand-side coefficient function $g$, and boundary function $\varphi$, respectively:
\begin{enumerate}
	\item[$(\mathbf{A_1})$]  $\Omega$ is a   uniformly convex domain of class $C^2$, where the uniform convexity of a domain is a concept from differential geometry, which states that all  principal curvatures  of the boundary  $\partial\Omega$, denoted by $\kappa_1,\cdots,\kappa_{n-1}$, are bounded away from zero.  
	\item[$(\mathbf{A_2})$]  $F(\lambda_1,\cdots,\lambda_n)\in C(\mathbb{R}^n)$;  $F(\lambda_1,\cdots,\lambda_n)$ satisfies the semi-elliptic condition, that is,
	\begin{equation}\label{semi-elliptic}
		F(\lambda_1(A),\cdots \lambda_n(A))\geq 	F(\lambda_1(B),\cdots \lambda_n(B))
	\end{equation}
	for any 
	$n$-order  
	positive semi-definite 
	matrices $A$ and $B$ with $A\geq B$ (this symbol means that $A-B$ is positive semi-definite); 
	and  either one of the following two boundedness conditions holds: 
	\begin{enumerate}
		\item[(i)] there exist
		some constants $\Lambda_0>0$, $a\geq 0$ and $b> 0$ such that
		\begin{equation*}
			F\left(\lambda_1(D^2 u),\cdots,\lambda_n(D^2 u)\right)\geq \Lambda_0(\lambda_{\min}(D^2 u))^{a}(\lambda_{\max}(D^2 u))^{b},
		\end{equation*}
		\item[(ii)] there exist
		some constants $\Lambda_0>0$, $a\geq 0$ and $b> 0$ such that
		\begin{equation*}
			F\left(\lambda_1(D^2 u),\cdots,\lambda_n(D^2 u)\right)\leq \Lambda_0(\lambda_{\min}(D^2 u))^{a}(\lambda_{\max}(D^2 u))^{b},
		\end{equation*}
	\end{enumerate}
	where $\lambda_{\min}(D^2 u):=\min\{\lambda_1(D^2 u),\cdots,\lambda_n(D^2 u)\}$ and
	$\lambda_{\max}(D^2 u):=\max\{\lambda_1(D^2 u),\cdots,\lambda_n(D^2 u)\}$. 
	\item[$(\mathbf{A_3})$] $g(x)\in C(\overline{\Omega})$; and there exist
	some constants $0<m_0\leq M_0$ such that
	\begin{equation*}
		m_0\leq g(x)\leq M_0,\ \ \forall x\in\overline{\Omega}.
	\end{equation*}
	\item[$(\mathbf{A_4})$] $\varphi(x)\in C(\partial\Omega)$ is a  function defined on $\partial\Omega$; and $\varphi$ has at least one continuous convex extension to $\overline{\Omega}$, implying that there exists a convex function $\varphi_*\in C(\overline{\Omega})$ such that $\varphi_*\big|_{\partial\Omega}=\varphi$.
\end{enumerate}

Under these assumptions,  \eqref{equation} can describe a class of fully nonlinear elliptic equations, and 
the right-hand side of   \eqref{equation} can be degenerate or singular on the boundary $\partial\Omega$. 
It is evident from the assumption $(\mathbf{A_3})$ that if $\alpha>0$,  then $g(x)(d(x))^{\alpha}\to 0$ as $x\to\partial\Omega$, indicating the degenerate case; if $\alpha<0$, then $g(x)(d(x))^{\alpha}\to +\infty$ as $x\to\partial\Omega$, reflecting the singular case.

For the equation \eqref{equation} with a  degenerate or singular right-hand side at the boundary, 
the tangential and normal eigenvalues of $D^2u$ might exhibit degeneracy or singularity on the boundary at different speeds.  In fact,  we have the following eigenvalue analysis on the boundary, given that  $\partial\Omega$ is uniformly convex.  On one hand,  when the right-hand side of \eqref{equation} degenerates, it follows that the normal eigenvalues also degenerate,  whereas the tangential eigenvalues typically do not. In this case, all tangential eigenvalues are  larger than all normal eigenvalues as  the boundary is approached. On the other hand, when the right-hand side of \eqref{equation} becomes singular, it holds that the normal eigenvalues are definitely singular, while the tangential eigenvalues may be singular (if the right-hand side exhibits strong singularity) or non-singular (if the right-hand side has weak singularity). In this scenario, all normal eigenvalues are larger than all tangential eigenvalues near the boundary. Generally speaking,  
the eigenvalues of $D^2u$  always can be split into two groups with distinct orders of speed:  $l$ small
 and $n-l$ large  for some
 $l\in\{1,\cdots,n-1\}$, i.e.
 \begin{equation}\label{eq:eigen-1} 
 	\lambda_1(D^2u)\sim\cdots\sim\lambda_l(D^2u)\ll\lambda_{l+1}(D^2u)\sim\cdots\sim\lambda_n(D^2u),
 \end{equation}
 and then we have 
\begin{equation}\label{eq:eigen-2} 
	\{\lambda_1(D^2u),\cdots,\lambda_n(D^2u)\}\sim\{\underbrace{\lambda_{\min}(D^2u),\cdots,\lambda_{\min}(D^2u)}_{l},\underbrace{\lambda_{\max}(D^2u),\cdots,\lambda_{\max}(D^2u)}_{n-l}\},
\end{equation}
where the symbol ``$\sim$''  
 signifies identical orders of speed.

Since $u$ is a convex function, there are several usual  operators that meet the requirements of $(\mathbf{A_2})$. For instance, we observe the following: 
	 For any $t,t_1,\cdots,t_n> 0$, both the operators $(\lambda_1(D^2u))^{t}+\cdots+(\lambda_n(D^2u))^{t}$ and $(\lambda_1(D^2u))^{t_1}\cdots(\lambda_n(D^2u))^{t_n}$ fulfill $(\mathbf{A_2})$ due to the following inequalities:
	\begin{align*}
		&(\lambda_{\max}(D^2u))^{t}\leq	(\lambda_1(D^2u))^{t}+\cdots+(\lambda_n(D^2u))^{t}\leq n(\lambda_{\max}(D^2u))^{t},\\
		&(\lambda_{\min}(D^2u))^{\sum_{i=1}^nt_i-\min\{t_1,\cdots,t_n\}}(\lambda_{\max}(D^2u))^{\min\{t_1,\cdots,t_n\}}
		\leq (\lambda_1(D^2u))^{t_1}\cdots(\lambda_n(D^2u))^{t_n}\leq (\lambda_{\max}(D^2u))^{\sum_{i=1}^nt_i}.
	\end{align*}
	In general, if 
	\begin{equation*}
		F(\lambda_1(D^2u),\cdots,\lambda_n(D^2u))=\sum_{(t_1,\cdots,t_n)\in S} C(t_1,\cdots,t_n)(\lambda_1(D^2u))^{t_1}\cdots (\lambda_n(D^2u))^{t_n}
	\end{equation*}
	is well-defined with $C(t_1,\cdots,t_n)>0$ for all $(t_1,\cdots,t_n)\in S$ (which represents an index set), then 
	this $F$ also satisfies $(\mathbf{A_2})$.

Another classical type of examples for $F$ is the $k$-Hessian operator.  For any $k\in\{1,\cdots,n\}$,  
the $k$-Hessian operator is given by 
	the $k$-th elementary symmetric polynomial of eigenvalues of the Hessian:
		\begin{equation*}
		\sigma_k(D^2u)=\sum_{1\leq i_1<\cdots<i_k\leq n}\lambda_{i_1}(D^2u)\cdots \lambda_{i_k}(D^2u).
	\end{equation*} 
 Indeed, one can check
\begin{equation*}
	\binom{n-1}{k-1}(\lambda_{\min}(D^2u))^{k-1}	\lambda_{\max}(D^2u)\leq F(\lambda_1(D^2u),\cdots,\lambda_n(D^2u))\leq \binom{n}{k}(\lambda_{\max}(D^2u))^k.
\end{equation*}
 Notably,
   $\sigma_1(D^2u)=\Delta u$ corresponds to the Laplace operator, while $\sigma_n(D^2u)=\det D^2u$ corresponds to the Monge-Amp\`ere operator.  The $k$-Hessian  operator is fully nonlinear when $k \geq 2$, and is elliptic when
 restricted to $k$-admissible functions.  
 
 This leads to a variety of special cases and geometric  applications of \eqref{equation}. For example, the  $k$-Hessian type equations \cite{Wang-2009,Jian} in the form of
 \[\sigma_k(D^2u)=g(x)(d(x))^{\alpha}\]
 can be related to the 
 $k$-Yamabe problems in conformal geometry \cite{Wang-2009,Viaclovsky}   and special Lagrangian equations in calibrated geometry \cite{Yuan}. Notably, the following
 quadratic Hessian equation (sigma-2 equation)  \cite{Yuanyu}  is also included:
 \[\sigma_2(D^2u)=1.\]
When $k=n$, the Monge-Ampère type equations \cite{Figalli}, taking the form 
  \[\det D^2u=g(x)(d(x))^{\alpha},\]
figure in 
  geometric optics problems \cite{G-Huang}.  More specially, a typical example of \eqref{equation} arising in affine geometry  is 
 $$\det D^2u=1, $$
for which the Legendre transform of solution in the entire space, due to the J\"orgens-Calabi-Pogorelov theorem \cite{J,C,P},  constructs a complete parabolic affine sphere. 
Various examples and  widespread applications can also be found in   \cite{Figalli,Wang-2009,G-T}.

\subsection{Review of previous results} 

During the last half-century, 
existence and nonexistence theories of solutions to fully nonlinear elliptic equations have been investigated by many researchers. We  refer the 
readers to \cite{Liyanyan,Caffarelli-Nirenberg-Spruck-I,Caffarelli-Nirenberg-Spruck-III,Cheng-Yau,Jiang-Trudinger-Yang,Urbas90,Guanbo,Jian-Wang1,Guan-Jian,Wang-Jiang,Jian} for example for earlier works on related topics. 
In  Theorem 3 of  \cite{Cheng-Yau}, 
the existence of a unique continuous convex generalized solution  to the   problem 
 \begin{align*}
 	\det D^2u &= f(x)\quad \text{in} \; \Omega \\
 	u &= \varphi \quad \ \ \  \text{on} \; \partial\Omega
 \end{align*}
was established by Cheng and Yau, where $\Omega$ is a bounded $C^2$ strictly convex domain, $\varphi\in C^2(\partial\Omega)$, and 
 $f$ belongs to class $C^k$ ($k\geq3$) and behaves  near the boundary as follows:
 \begin{equation*}
 	0<f(x)\leq A (d(x))^{\alpha}
 \end{equation*}
 for some constants $A>0$ and $\alpha>-n-1$. 
In Theorem 2 of    \cite{Caffarelli-Nirenberg-Spruck-III}, given the smoothness and some appropriate conditions on  the domain $\Omega$,  operator $F$, right-hand-side function $f$ and boundary function $\varphi$ respectively,  Caffarelli, Nirenberg and Spruck obtained
 the existence of a unique smooth admissible solution  to the   problem 
 \begin{align*}
F\left(\lambda_1(D^2 u),\cdots,\lambda_n(D^2 u)\right) &= f(x)\quad \text{in} \; \Omega \\
 	u &= \varphi \quad \ \ \ \text{on} \; \partial\Omega,
 \end{align*}
where it is noteworthy that the  right-hand-side function $f$ 
is required to be smooth and positive, which also means that the equation cannot be degenerate or singular. 
In comparison, the equations considered in this paper  
may exhibit degeneracy or singularity, and  
we aim to  
comprehensively classify the existence and nonexistence of solutions based on the values of $\alpha$.
We note that the continuity of the right-hand-side coefficient function $g(x)$ in  $(\mathbf{A_3})$ does not guarantee the existence of a classical solution to the problem  \eqref{equation}-\eqref{bcondition}. Instead,  
 our focus will be on the existence and nonexistence of viscosity solutions to this problem  rather than classical solutions.  
 
 The Perron method is known as a powerful technique used to establish the existence of viscosity solutions. This method   might ﬁrst appear in \cite{Ishii-1}, introduced by Ishii for the study of Hamilton-Jacobi equations, and  shortly afterwards, was 
applied to infer existence of viscosity solutions of the Dirichlet problem of fully nonlinear elliptic equations in \cite{Ishii-2}.  
 The main idea of this method involves constructing a family of 
 viscosity sub-solutions or  viscosity super-solutions  that serve as candidates for the viscosity solution. By considering the supremum of  
  viscosity sub-solutions or the infimum of  
  viscosity super-solutions, one can identify a function that satisfies the  
   conditions of being a viscosity solution. 
  The whole process also heavily  relies on using 
 comparison principle. The Perron method, comparison principle as well as  constructions of suitable classical sub-solutions and  super-solutions  
 will play an important role in this paper.

This paper also builds upon and extends our earlier works \cite{Li-Li-2022,Li-Li-2025}.  
In \cite{Li-Li-2022}, assuming the existence of solution, we derived the \textit{a priori} H\"older regularity estimates for the Dirichlet problem of fully nonlinear elliptic equations:
\begin{equation*}\begin{split} 
		F(\lambda_1(D^2u),\cdots,\lambda_n(D^2u))&=f(x,u,\nabla u)\text{\hspace{0.2cm} in } \Omega,\\
		u&=0\text{\hspace{1.7cm} on } \partial\Omega.
\end{split}\end{equation*} 
These estimates were obtained under certain structure conditions on the domain $\Omega$, operator $F$ and right-hand-side function $f$ respectively. In particular,   
  $f$ was required to satisfy the following bound:
\begin{equation*} 
	\begin{split}
		&\text{for some constants $A \geq 0$, $\alpha\in\mathbb{R}$, $\gamma\in\mathbb{R}$, $\beta\geq \max\{n+1,n+1+ \gamma\}$, there holds}\\
		& 0\leq f(x,z,q)\leq
		A(d(x))^{\beta-n-1}|z|^{-\alpha}(1+|q|^2)^{\frac{\gamma}{2}},\ \ \forall
		(x,z,q)\in\Omega	\times(-\infty,0) \times \mathbb{R}^n,
\end{split}\end{equation*}
where $	\beta-n-1\geq 0$ always holds, implying that the exponent of $d(x)$ must be nonnegative. In contrast to \cite{Li-Li-2022}, the results in this paper will thoroughly analyze all cases  
involving the exponent of $d(x)$ for the problem \eqref{equation}-\eqref{bcondition}. Furthermore, we will establish a nonexistence result in this paper  
when 
the exponent of $d(x)$ is sufficiently small, specifically when it is no greater than $-a-2b$, i.e. $\alpha\leq -a-2b$ based on the  parameter notations in  \eqref{equation} and $(\mathbf{A_2})$.

Our recent work \cite{Li-Li-2025}  relates to the study of viscosity solution to the Dirichlet problem of Monge-Amp\`ere type equations:
\begin{equation}\begin{split}\label{eq:2025}
		\det D^2u&=f(x,u,\nabla u)\text{\hspace{0.2cm} in } \Omega,\\
		u&=0\text{\hspace{1.7cm} on } \partial\Omega. 
\end{split}\end{equation}
Under certain structure conditions on the domain $\Omega$ and right-hand-side function $f$, 
we established the existence, uniqueness and interior regularity of the viscosity solution to this problem. We recall  that  $f$ therein satisfies that
\begin{equation*} 
	\begin{split}
		&\text{for some constants $A > 0$, $\alpha\in\mathbb{R}$, $\beta\geq n+1$, $\gamma<\min\{n+\alpha,\beta-n+1\}$,  there holds}\\
		& 0< f(x,z,q)\leq
		A(d(x))^{\beta-n-1}|z|^{-\alpha}(1+|q|^2)^{\frac{\gamma}{2}},\ \ \forall
		(x,z,q)\in\Omega	\times(-\infty,0) \times \mathbb{R}^n, 
\end{split}\end{equation*} 
where  $\beta-n-1\geq 0$ also means that the exponent of $d(x)$ needs to be nonnegative.  
Unlike \cite{Li-Li-2025}, for the problem \eqref{equation}-\eqref{bcondition}, given that $\partial\Omega\in C^2$ is uniformly convex, we will  construct  different auxiliary functions as the  exponent of $d(x)$ varies, and moreover
in the critical case $\alpha=-b$ based on the  parameter notations in \eqref{equation} and $(\mathbf{A_2})$,  we will obtain a global $\log$-type estimate for the viscosity solution.

\subsection{Main results and outline of this paper}

According to the assumption $(\mathbf{A_4})$,   one can extend $\varphi$ from $\partial\Omega$ to $\overline{\Omega}$. We are going to  uniquely extend  $\varphi$ to  $\overline{\Omega}$ through one Monge-Amp\`ere equation. 
In fact, by the existence result for Monge-Amp\`ere equations with continuous boundary data (see Theorem 2.14 in  \cite{Figalli}), 
there exists a unique convex function $\psi\in C(\overline{\Omega})$ that solves the   Dirichlet problem
\begin{equation}\label{eq:MA-1}
	\begin{split}
		\det D^2\psi&=0\ \ \text{ in }\Omega,\\
		\psi&=\varphi\ \ \text{ on }\partial\Omega.
	\end{split}
\end{equation}

\begin{Remark}\label{remark:extension}
	In some sense, the extension function	$\psi\in C(\overline{\Omega})$ 
	can be regarded as a weak solution to the  following  homogeneous equation with Dirichlet boundary condition:
	\begin{equation}\label{eq:MA-3}
		\begin{split}
			\Lambda_0\left(\lambda_{\min}(D^2\psi)\right)^{a}\left(\lambda_{\max}(D^2\psi)\right)^{b}&=0\text{\hspace{0.2cm} in } \Omega,\\
			\psi&=\varphi\text{\hspace{0.2cm} on } \partial\Omega,
		\end{split}
	\end{equation}
	This observation will be easy to explain  
	if we assume  $\psi\in C^2(\Omega)$. 
	We first note that
	\[\lambda_1(D^2\psi)\cdots \lambda_n(D^2\psi)=\det D^2\psi=0 \ \ \text{ in }\Omega. \]
	Since $\psi$ is convex, it holds  in $\Omega$ that
	\[\lambda_i(D^2\psi)\geq 0,\ \ \forall i=1,\cdots,n.\]
	Now we can infer that 
	\[\lambda_{\min}(D^2\psi)=0\ \ \text{ in }\Omega. \] 
	Therefore, $\psi$ actually satisfies \eqref{eq:MA-3} 
	for any constants $\Lambda_0>0$, $a> 0$ and $b> 0$.
\end{Remark}

Now we are ready to state our main results. A series of existence results  can be elaborated as the following theorem. Global estimates for viscosity solutions via distance function estimates are also obtained. 
 
\begin{Theorem}[Existence Theorem]\label{thm1}
	Under the assumptions $(\mathbf{A_1})$,  $(\mathbf{A_2})$-$(\mathrm{i})$,  $(\mathbf{A_3})$ and $(\mathbf{A_4})$, we have a list of existence and estimation conclusions based on varying values of $\alpha$:
	\begin{enumerate}[(i)]
		\item If $\alpha\in (-b,+\infty)$, then the problem \eqref{equation}-\eqref{bcondition} admits a viscosity solution $u(x)\in C(\overline{\Omega})$, and there exists some constant $\overline{C_0}>0$ such that 
		\begin{equation*}
			u(x)\geq \psi(x)-\overline{C_0}\,d(x), \ \ \forall x\in\overline{\Omega}.
		\end{equation*}

\item If $\alpha=-b$, then the problem \eqref{equation}-\eqref{bcondition} admits a viscosity solution $u(x)\in C(\overline{\Omega})$, and there exists some constants $\overline{C_0},\,\overline{C_1}>0$ such that 
\begin{equation*}
u(x)\geq \psi(x)-\overline{C_0}\,d(x)\left(-\log\left(\overline{C_1}\,d(x)\right)\right)^{\frac{b}{a+b}}, \ \ \forall x\in\overline{\Omega}.
\end{equation*}

\item If $\alpha\in(-a-2b,-b)$, then the problem \eqref{equation}-\eqref{bcondition} admits a viscosity solution $u(x)\in C(\overline{\Omega})$, and there exists some constant $\overline{C_0}>0$ such that 
\begin{equation*}
	u(x)\geq \psi(x)-\overline{C_0}\left(d(x)\right)^{\frac{a+2b+\alpha}{a+b}}, \ \ \forall x\in\overline{\Omega}.
\end{equation*}
	\end{enumerate} 
\end{Theorem}

 Our second theorem addresses the issue of nonexistence: 
 
	\begin{Theorem}[Nonexistence Theorem]\label{thm2}
		Under the assumptions $(\mathbf{A_1})$,  $(\mathbf{A_2})$-$(\mathrm{ii})$,  $(\mathbf{A_3})$ and $(\mathbf{A_4})$, we have one nonexistence conclusion: 	
		if $\alpha\in(-\infty,-a-2b]$, then the problem \eqref{equation}-\eqref{bcondition} 
		does not admit any viscosity solutions, and consequently, it also does not admit classical solutions.
 \end{Theorem}

The following corollary follows directly from Theorems \ref{thm1} and   \ref{thm2}: 

\begin{Corollary}[Special Case]\label{thm3}
	Under the assumptions   $(\mathbf{A_1})$, $(\mathbf{A_3})$ and $(\mathbf{A_4})$, 
	all the conclusions in Theorems \ref{thm1} and   \ref{thm2} 
	remain valid for the following problem:
	\begin{align*}
		\Lambda_0\left(\lambda_{\min}(D^2u)\right)^{a}\left(\lambda_{\max}(D^2u)\right)^{b}&=g(x)(d(x))^{\alpha}\text{\hspace{0.2cm} in } \Omega,\stepcounter{equation}\tag{\theequation}\label{equation'}\\
		u&=\varphi\text{\hspace{1.7cm} on } \partial\Omega,\stepcounter{equation}\tag{\theequation}\label{bcondition'}
	\end{align*}
	for any constants $\Lambda_0>0$, $a\geq 0$ and $b> 0$.
\end{Corollary}

\begin{Remark}
	For the problem  \eqref{equation}-\eqref{bcondition}, 
by combining Theorems \ref{thm1} and   \ref{thm2},	 we achieve a complete classification for   $\alpha\in (-\infty,+\infty)$, i.e. the exponent of $d(x)$ in  \eqref{equation}.  	For the problem  \eqref{equation'}-\eqref{bcondition'}, 
this configuration also provides a complete classification on the  existence and nonexistence results for the viscosity solution. Intuitively, when $\alpha$ is large (tending more towards the degenerate case), viscosity solutions exist. In contrast, when $\alpha$ is small (tending more towards the singular case), viscosity solutions do not exist. The critical value that distinguishes these two scenarios is $-a - 2b$. We will explain these intuitions in Section \ref{sec:discuss}. 
\end{Remark}

The rest of this paper is organized as follows. In Section \ref{Sec:Pre}, we present basic concepts,  preliminary results and necessary estimates for the problem \eqref{equation}-\eqref{bcondition}. In Section \ref{Sec:lemma2}, we establish a key lemma that provides a transition from classical sub-solutions to viscosity solutions. Based on this transition and constructions of suitable classical sub-solutions, Section \ref{Sec:Case 1-3} is dedicated to the proofs for existence results and global estimates of viscosity solutions  in Theorem \ref{thm1}.  We devote Section \ref{Sec:Case 5} to demonstrate the nonexistence result in Theorem \ref{thm2} through a contradiction by constructing a suitable classical super-solution. Finally, some intuitions and discussions regarding our results are provided in Section \ref{sec:discuss}.

\section{Preliminaries}\label{Sec:Pre}

Throughout this paper, we adhere to the following convention:  for any function $\Phi$, we always use the notation $D^2\Phi>(<)\, 0$ to represent that the Hessian matrix $(D^2\Phi)$ is positive (negative) definite and use the notation $D^2\Phi\geq(\leq)\, 0$ to represent that the Hessian matrix $(D^2\Phi)$ is positive (negative) semi-definite. 

\subsection{Classical solutions and viscosity solutions}

Let us begin this subsection by clarifying the concepts of classical solutions and viscosity solutions to the problem \eqref{equation}-\eqref{bcondition}.

\begin{Definition}\label{def:class}
		Let $u\in C^2(\Omega)\cap C(\overline{\Omega})$ be a convex function. 
	We refer to $u$ as a classical solution (sub-solution, super-solution) to \eqref{equation} over $\Omega$ if  
	\begin{equation*}
		F\left(\lambda_1(D^2 u),\cdots,\lambda_n(D^2 u)\right) =\,(\geq,\, \leq)\, g(x)(d(x))^{\alpha}\text{\hspace{0.2cm} in } \Omega.
	\end{equation*}
  $u$ is called  a classical solution to the  problem  \eqref{equation}-\eqref{bcondition}  over $\Omega$ if it is a classical solution to \eqref{equation}  over $\Omega$ and satisfies \eqref{bcondition} on $\partial\Omega$.  Moreover, we say that $u$ is a classically strict  sub-solution (super-solution)  to \eqref{equation} over $\Omega$ if 
  \[	F\left(\lambda_1(D^2 u),\cdots,\lambda_n(D^2 u)\right)>(<)\, g(x)(d(x))^{\alpha}\text{\hspace{0.2cm} in } \Omega.\] 
\end{Definition}
 
\begin{Definition}\label{def:vis}
	Let $u\in C(\overline{\Omega})$ be a convex function. 
   We refer to $u$ as a viscosity sub-solution (super-solution)  to \eqref{equation}  over $\Omega$ 
if  
for any point $x_0\in\Omega$,  any open neighborhood $U(x_0)\subset\Omega$ and any convex function $\phi\in C^2(U(x_0))$ satisfying
\[(\phi-u)(x)\geq\, (\leq)\, (\phi-u)(x_0),\ \ \forall x\in U(x_0),\]
there holds 
\begin{equation*}
F\left(\lambda_1(D^2 \phi(x_0)),\cdots,\lambda_n(D^2 \phi(x_0))\right)\geq\, ( \leq)\, g(x_0)(d(x_0))^{\alpha}.
\end{equation*}
 $u$ is said to be  a viscosity  solution to  \eqref{equation}  over $\Omega$ if it is both a viscosity sub-solution and a viscosity super-solution to \eqref{equation}  over $\Omega$. 
   $u$ is called    a viscosity  solution to the   problem  \eqref{equation}-\eqref{bcondition}  over $\Omega$ if it is a viscosity solution to \eqref{equation}  over $\Omega$ and satisfies \eqref{bcondition} on $\partial\Omega$. 
\end{Definition}

See the lectures of Crandall collected in \cite{Crandall}   for a general theory on viscosity solutions. 
Next, we will demonstrate the  equivalent condition and  comparison principle for viscosity solutions  to \eqref{equation} as the following two lemmas.

\begin{Lemma}\label{lemma:equiv}
	Under the semi-elliptic assumption in  $(\mathbf{A_2})$,  
	if $u\in C^2(\Omega)$ is a convex function, then the following two statements are equivalent:
	\begin{enumerate}
	\item[(i)] $u$ is a viscosity solution (sub-solution, super-solution) to \eqref{equation}  over $\Omega$;
	\item[(ii)] $u$ is a classical solution (sub-solution, super-solution) to \eqref{equation}  over $\Omega$.
	\end{enumerate}
\end{Lemma}
\begin{proof}
The proofs for all three cases are not easily unified. However, we can derive the results for the sub-solution case and the super-solution case separately, and combining these two results directly leads to the conclusion for the solution case. 
	
We first infer (ii) from (i). 
Let $u$ be a viscosity sub-solution (super-solution) to \eqref{equation}  over $\Omega$. For any $x_0\in\Omega$ and any $x\in U(x_0)$,  we take $\phi(x)=u(x)$. Since $u\in C^2(\Omega)$, there holds $\phi\in C^2(U(x_0))$. As a result of $\phi(x)=u(x)$,  
we have $\phi(x)-u(x)=0=\phi(x_0)-u(x_0)$ and therefore 
\[(\phi-u)(x)\geq\,(  \leq)\,(\phi-u)(x_0).\]
By Definition \ref{def:vis}, it follows that 
\begin{equation*}
	F\left(\lambda_1(D^2 \phi(x_0)),\cdots,\lambda_n(D^2 \phi(x_0))\right)\geq\,(\leq)\, g(x_0)(d(x_0))^{\alpha},
\end{equation*}
which together with $\phi(x)=u(x)$ implies that
\begin{equation*}
	F\left(\lambda_1(D^2 u(x_0)),\cdots,\lambda_n(D^2 u(x_0))\right)\geq\,( \leq)\, g(x_0)(d(x_0))^{\alpha}.
\end{equation*}
Since $x_0\in\Omega$ is arbitrary, we infer that $u$ is a classical sub-solution (super-solution) to \eqref{equation}  over $\Omega$.

From (ii) to (i), we let $u$ to be a classical sub-solution (super-solution) to \eqref{equation}  over $\Omega$. Then for any $x_0\in\Omega$, we have
\begin{equation*}
	F\left(\lambda_1(D^2 u(x_0)),\cdots,\lambda_n(D^2 u(x_0))\right)\geq(\leq)\, g(x_0)(d(x_0))^{\alpha}.
\end{equation*} 
For any convex function $\phi\in C^2(U(x_0))$ satisfying
\[(\phi-u)(x)\geq(\leq)\,(\phi-u)(x_0),\ \ \forall x\in U(x_0),\]
it is clear that $\phi-u$ takes its minimum (maximum) at $x_0$, which yields
$$D(\phi-u)\big|_{x_0}=0,\ \ D^2(\phi-u)\big|_{x_0}\geq(\leq)\,0.$$
Due to \eqref{semi-elliptic}, $F(\lambda_1,\cdots,\lambda_n)$ is semi-elliptic, we further obtain
$$D\phi(x_0)=D u(x_0),\ \ F\left(\lambda_1(D^2 \phi(x_0)),\cdots,\lambda_n(D^2 \phi(x_0))\right)\geq(\leq)\, F\left(\lambda_1(D^2 u(x_0)),\cdots,\lambda_n(D^2 u(x_0))\right).$$
It then follows that 
\[F\left(\lambda_1(D^2 \phi(x_0)),\cdots,\lambda_n(D^2 \phi(x_0))\right)\geq(\leq)\, F\left(\lambda_1(D^2 u(x_0)),\cdots,\lambda_n(D^2 u(x_0))\right)\geq  (\leq)\, g(x_0)(d(x_0))^{\alpha},\]
which means that $u$ is a  viscosity sub-solution (super-solution) to \eqref{equation}  over $\Omega$.

 The proof of this lemma is now complete. 
\end{proof}

	\begin{Lemma}[Comparison Principle]\label{lemma1}
  If $u\in C(\overline{\Omega})$ is a viscosity solution  to \eqref{equation}  over $\Omega$ and $W(x)\in C^2(\Omega)\cap C(\overline{\Omega})$ is convex function satisfying 
\begin{align*}
		F\left(\lambda_1(D^2 W(x)),\cdots,\lambda_n(D^2 W(x))\right)&>(<)\, g(x)(d(x))^{\alpha},\ \ \ \  \forall x\in \Omega,\stepcounter{equation}\tag{\theequation}\label{CP-1}\\
		W(x)&\leq(\geq)\, u(x),\ \ \ \  \forall x\in \partial\Omega,\stepcounter{equation}\tag{\theequation}\label{CP-2}
\end{align*}
then 
\begin{equation}\label{CP-3}
W(x)\leq(\geq)\, u(x), \ \ \ \  \forall x\in\overline{\Omega}.
\end{equation}
Here, based on Definition \ref{def:class},   the condition \eqref{CP-1} is also equivalent to the condition that 
$W$ is a  classically strict  sub-solution (super-solution) to \eqref{equation} over $\Omega$.
\end{Lemma}
\begin{proof}
	To prove \eqref{CP-3}, 
it suffices to prove that $$\displaystyle\max_{x\in\overline{\Omega}}\,(W-u)(x)\leq 0\ \ (\text{ or }\displaystyle\min_{x\in\overline{\Omega}}\,(W-u)(x)\geq 0).$$
The proof is by contradiction. 
Suppose there exists a point  $x_0\in\overline{\Omega}$ such that
\begin{equation*}
\max_{x\in\overline{\Omega}}\,(W-u)(x)=(W-u)(x_0)> 0\ \ (\text{ or }\min_{x\in\overline{\Omega}}\,(W-u)(x)=(W-u)(x_0)< 0).
\end{equation*}   
Due to \eqref{CP-2}, we have  $(W-u)\big|_{\partial\Omega}\leq(\geq)\, 0$, and hence $x_0\in\Omega$. Now $x_0$ is the maximum (minimum) point of $(W-u)(x)$ in the interior of $\Omega$. Since   $u$ is a viscosity solution to \eqref{equation} over $\Omega$, then $u$ is also a viscosity super-solution (sub-solution) to \eqref{equation} over $\Omega$. 
Given that $W\in C^2(\Omega)$, by Definition \ref{def:vis}, we can further derive
\begin{equation*}
	F\left(\lambda_1(D^2 W(x_0)),\cdots,\lambda_n(D^2 W(x_0))\right) \leq(\geq)\,   g(x_0)(d(x_0))^{\alpha},
\end{equation*}
which contradicts our condition \eqref{CP-1}. We have thus completed the proof of the lemma.
\end{proof}

\subsection{Construction of approximation functions}

 In this subsection, we   construct a family of Monge-Amp\`ere equations with continuous boundary data, 
 generating a family of approximation functions via their solutions.

\begin{Lemma}[Approximation function  $\psi_k$ over $\overline{\Omega}$]\label{lemma:approximation}
Let $\Omega$ be a bounded and  strictly convex domain in $\mathbb{R}^n$.	Let $\varphi$ satisfy the assumption $(\mathbf{A_4})$. Then there exists a family of convex functions $\{\psi_k\}_{k\in\mathbb{N}^+}$  such that 
	\begin{enumerate}[(i)] 
		\item $\psi_k\in C^\infty(\Omega)\cap C(\overline{\Omega})$;
		\item $\psi_k\big|_{\partial\Omega}=\varphi$;
		\item  $\psi_k$ converges locally  uniformly to $\psi$ over $\Omega$.
	\end{enumerate}
\end{Lemma}

\begin{proof}

For any $k\in\mathbb{N}^+$,  let  $\psi_k$ be the solution to  the nonhomogeneous Dirichlet problem
\begin{equation*} 
	\begin{split}
		\det D^2\psi_k&=\tfrac{1}{k}\ \ \text{ in }\Omega,\\
		\psi_k&=\varphi\ \ \text{ on }\partial\Omega.
	\end{split}
\end{equation*}
Since the constant $\frac{1}{k}$  can be considered  a smooth function, then we obtain that 
$\psi_k\in C^\infty(\Omega)\cap C(\overline{\Omega})$ is a convex function. It is clear that $\frac{1}{k}\to 0$ as $k\to\infty$, and hence for Borel measures, we have $\frac{1}{k}dx\to^* 0dx$  as $k\to\infty$. By the stability result for Monge-Amp\`ere equations with continuous boundary data (see Proposition 2.16 in  \cite{Figalli}),  it follows that 
$\psi_k$ converges  locally uniformly  in $\Omega$ to the unique solution $\psi$ of the  Dirichlet problem \eqref{eq:MA-1}.  Therefore, we have  proved the lemma.
\end{proof}

\begin{Remark}\label{remark:phi-k}
	For any $k\in\mathbb{N}^+$, 
since $\psi_k$ constructed in Lemma \ref{lemma:approximation} is a smooth convex function, then
\begin{equation*}
\text{$D^2\psi_k$ is a positive semi-definite matrix, i.e. $D^2\psi_k\geq 0$.}
\end{equation*}  
Moreover, it follows that for any nonzero vector $\bm{\xi}$, we have $\bm{\xi}^TD^2\psi_k\,\bm{\xi}\geq 0$.
\end{Remark}

\subsection{Distance function estimates}

For $\mu>0$, let us set $$\Gamma_{\mu}=\{x\in\overline{\Omega}:\ d(x)<\mu\}.$$ 

As a  consequence of 
 Lemma 14.16 and  Lemma 14.17 in \cite{G-T}, we have the following lemma that relates the $C^2$ smoothness of $d(x)$ in $\Gamma_{\mu}$ to that of the boundary $\partial\Omega$ and expresses
 the gradient and Hessian  of $d(x)$ at points close to $\partial\Omega$.

 \begin{Lemma}[Near boundary properties of distance function $d(x)$]\label{lemma:GT}
 	Under the assumption $(\mathbf{A_1})$, we have the following conclusions:  
 	\begin{enumerate}[(i)]
 		\item There exists some constant $\mu_0=\mu_0(\Omega)>0$
 		such that
 		\begin{equation*}
 			d(x)\in C^2(\Gamma_{\mu_0});
 		\end{equation*}
 		\item Let $x_0\in\Gamma_{\mu_0}$ and   $y_0\in\partial\Omega$ satisfy $d(x_0)=|x_0-y_0|$. Then by choosing some suitable coordinate system, the gradient and Hessian  of $d(x)$ at the point $x_0$ can be expressed as 
 		\begin{align*}
 		&Dd(x_0)=(\underbrace{0,\cdots,0}_{n-1},1),\\
 		&D^2d(x_0)=	 
 		\operatorname{diag}	\left(\frac{-\kappa_1(y_0)}{1-\kappa_1(y_0)d(x_0)},\cdots,	\frac{-\kappa_{n-1}(y_0)}{1-\kappa_{n-1}(y_0)d(x_0)} , 0\right),
 		\end{align*}
 	where $\kappa_1(y_0),\cdots,\kappa_{n-1}(y_0)$ represent all the $(n-1)$	principal curvatures of $\partial\Omega$ at $y_0$. 
 	\end{enumerate}
 \end{Lemma}
 \begin{proof}
Since $\Omega$ is a uniformly convex domain having non-empty boundary $\partial\Omega$, we know that $\Omega$ is a bounded domain. In fact, $\partial\Omega$ is an $(n-1)$-dimensional hypersurface, and 
all the $(n-1)$ principal curvatures of $\partial\Omega$ are bounded from below by a positive constant, which we take to be $c$ at present. Then for any unit tangent vector $X$, its Ricci curvature has the following positive lower bound:
\[\operatorname{Ric}(X,X)\geq (n-2)c^2.\]
Using the Bonnet–Myers theorem, we infer that $\partial\Omega$ is compact and its diameter satisfies
\[\operatorname{diam}(\partial\Omega):=\sup_{p,q\in\partial\Omega}\operatorname{dist}(p,q)\leq\frac{\pi}{c},\]
which together with the uniform convexity of $\Omega$ further implies that $\Omega$ is bounded.

The fact that $\Omega$ is bounded and of class $C^2$ yields that $\partial\Omega$ satisfies the following uniform interior sphere condition: there exists a constant $\mu_0=\mu_0(\Omega)>0$ such that  for any $y_0\in\partial\Omega$, there is a ball $B(y_0,r)$ with radius $r\geq \mu_0$ satisfying $\overline{B(y_0,r)}\cap (\mathbb{R}^n\setminus\Omega)=\{y_0\}$. Now it is clear  that $B(y_0,r)$ is tangent to $\partial\Omega$ at $y_0$ and lies entirely inside $\Omega$. Since all the principal curvatures of $\partial B(y_0,r)$ at $y_0$ are  $\frac{1}{r}$, then all the principal curvatures of $\partial\Omega$ at $y_0$ satisfy
\begin{equation}\label{eq:14.96''}
\kappa_i(y_0) \leq\frac{1}{r}\leq \frac{1}{\mu_0},\ \ \ \ \forall i=1,\cdots,n-1.
\end{equation}

Due to the $C^2$ regularity  of $\partial\Omega$ , for any $x\in\Gamma_{\mu_0}$, there exists a unique $y\in\partial\Omega$ such that $d(x)=|x-y|$. Let $N(y)$ and $T(y)$ denote respectively the unit inward normal to $\partial\Omega$ at $y$ 
 and the tangent hyperplane to $\partial\Omega$ at $y$. Let $\kappa_1(y),\cdots,\kappa_{n-1}(y)$ represent all the $(n-1)$	principal curvatures of $\partial\Omega$ at $y$. 
Now we choose a principal coordinate system at $y\in\partial\Omega$.  By a rotation of coordinates, we can assume that the $x_n$ coordinate axis lies in the direction $N(y)$, and the $x_1,\cdots,x_{n-1}$ coordinate axes lie along principal directions corresponding to $\kappa_1(y),\cdots,\kappa_{n-1}(y)$ at $y$. In fact, in some neighborhood $\mathcal{U}(y)$ of $y$, $\partial\Omega$ is  given by $x_n=h(x')$, where $x'=(x_1,\cdots,x_{n-1})$, $h\in C^2(T(y)\cap\mathcal{U}(y))$, 
$Dh(y')=0$ and 
\begin{equation}\label{eq:D2h}
D^2h(y')=
\operatorname{diag}(\kappa_1(y),\cdots, \kappa_{n-1}(y)).
\end{equation}
 The unit inward normal vector $\overline{N}(y')=N(y)$ at  $y=(y',h(y'))$ is then given by 
 \begin{align*}
 	&	N_i(y)=\frac{-D_ih(y')}{\sqrt{1+|Dh(y')|^2}},\ \ \forall i=1,\cdots,n-1,\stepcounter{equation}\tag{\theequation}\label{eq:Ni}\\
 	&N_n(y)=\frac{1}{\sqrt{1+|Dh(y')|^2}}.
 \end{align*}
 Hence  we have 
 \begin{equation*} 
 	D_j\overline{N}_i(y')=-\kappa_i(y)\delta_{ij},\ \ \forall i,j=1,\cdots,n-1.
 \end{equation*}

  We note that  $x$ is the closest point to $y$ in $\Omega$ and must lie along the inward normal direction as follows:
\begin{equation}\label{eq:14.96}
	x=y+t N(y),
\end{equation}
where $t\in\mathbb{R}$. For a fixed point $x_0\in \Gamma_{\mu_0}$, we let $y_0=y(x_0)$.  
Define a mapping $H=(H_1,\cdots,H_n):\,(T(y_0)\cap \mathcal{U}(y_0))\times\mathbb{R}\to\mathbb{R}^n$ by 
 \begin{equation*}
 H(y',t)=y+t N(y),\ \ \ \ y=(y',h(y')).
 \end{equation*}
 Clearly, $y_0\in\partial\Omega$ and $H\in C^1((T(y_0)\cap \mathcal{U}(y_0))\times\mathbb{R})$. 
Due to \eqref{eq:D2h} and \eqref{eq:Ni}, the gradient of $H$ at $(y_0',d(x_0))$ is given by
 \begin{equation*} 
 DH=\operatorname{diag}(1-\kappa_1(y_0)d(x_0),\cdots, 1-\kappa_{n-1}(y_0)d(x_0),1).
 \end{equation*} 
 Since $x_0\in\Gamma_{\mu_0}$, we also notice that 
\begin{equation}\label{eq:14.96'}
	d(x_0)<\mu_0.
\end{equation} 
By virtue of \eqref{eq:14.96''} and \eqref{eq:14.96'}, it follows that  the Jacobian of $H$ at $(y_0',d(x_0))$ is positive:
 \begin{equation*} 
 \det	DH=\left(1-\kappa_1d(x_0)\right)\cdots \left(1-\kappa_{n-1}d(x_0)\right)>0.
 \end{equation*}
Thus using the inverse mapping theorem implies that  $y'\in C^1(\mathcal{V}(x_0))$ for some neighborhood $\mathcal{V}(x_0)$, and by \eqref{eq:14.96}, there holds
 \[Dd(x)=N(y(x))=\overline{N}(y'(x))\in C^1(\mathcal{V}(x_0)),\ \ \forall x\in\mathcal{V}(x_0).\]
 Hence, $d\in C^2(\mathcal{V}(x_0))$ and thus we can derive $d\in C^2(\Gamma_{\mu_0})$, which leads us to (i).
 
We note that the gradient of $H^{-1}$ at $x_0$ is given by
 \begin{equation*} 
 	DH^{-1}=\operatorname{diag}\Big(\frac{1}{1-\kappa_1(y_0)d(x_0)},\cdots, \frac{1}{1-\kappa_{n-1}(y_0)d(x_0)},1\Big).
 \end{equation*} 
Now it is clear that 
 	\begin{align*}
 	&Dd(x_0)=N(y_0)=(\underbrace{0,\cdots,0}_{n-1},1),\\
 	&D_{in}d(x_0)=0,\ \ \forall i=1,\cdots,n,\\
 	&D_{ij}d(x_0)=D_jN_i\circ y(x_0)=D_l\overline{N}_i(y_0)D_jy_l(x_0)=\frac{-\kappa_i(y_0)\delta_{ij}}{1-\kappa_i(y_0)d(x_0)},\ \ \forall i,j=1,\cdots,n-1.
 \end{align*} 
This yields (ii) immediately. The proof of this lemma is now finished.
 \end{proof}

Based on Lemma \ref{lemma:GT}, 
though $d(x)$ is not uniformly convex near $\partial\Omega$, we can 
 show that $$-4\operatorname{diam}(\Omega)\cdot d(x)+(d(x))^2$$
is  uniformly convex near $\partial\Omega$. Our construction of this function is inspired by \cite{D-Z}, and   a detailed analysis of the distance function to a smooth hypersurface can be found in Section 5 therein.
\begin{Lemma}[Near boundary properties of function $-4\operatorname{diam}(\Omega)\cdot d(x)+(d(x))^2$]\label{lemma:inspiration}
Under the assumption $(\mathbf{A_1})$, there exists some constant $\mu_0=\mu_0(\Omega)>0$
 such that
 the function
$$-4\operatorname{diam}(\Omega)\cdot d(x)+(d(x))^2$$
is $C^2$ and uniformly convex over $\Gamma_{\mu_0}$. 
\end{Lemma}
\begin{proof}
Here, we adopt the notations used in the proof of Lemma \ref{lemma:GT}. 
	
By virtue of Lemma \ref{lemma:GT} (i), there exists $\mu_0=\mu_0(\Omega)>0$ such that $d(x)\in C^2(\Gamma_{\mu_0})$. Hence we have
$$-4\operatorname{diam}(\Omega)\cdot d(x)+(d(x))^2\in C^2(\Gamma_{\mu_0}).$$
For any $x_0\in\Gamma_{\mu_0}$, direct computations lead us to
\begin{align*}
&D\left( -4\operatorname{diam}(\Omega)\cdot d(x)+(d(x))^2\right)\big|_{x=x_0}=-4\operatorname{diam}(\Omega)\cdot Dd(x_0)+2d(x_0)\cdot Dd(x_0),\\
&D^2\left( -4\operatorname{diam}(\Omega)\cdot d(x)+(d(x))^2\right)\big|_{x=x_0}=\left(4\operatorname{diam}(\Omega)-2d(x_0)\right)\cdot \left(-D^2d(x_0)\right)+2Dd(x_0)\otimes Dd(x_0).
\end{align*}

Using Lemma \ref{lemma:GT} (ii), we note that 
\begin{align*}
&	-D^2d(x_0)=	\operatorname{diag}	\left(\frac{\kappa_1(y_0)}{1-\kappa_1(y_0)d(x_0)},\cdots,	\frac{\kappa_{n-1}(y_0)}{1-\kappa_{n-1}(y_0)d(x_0)} , 0\right),\\
&Dd(x_0)\otimes Dd(x_0)=\operatorname{diag}(\underbrace{0,\cdots,0}_{n-1},1).
\end{align*}
In view of  
the fact that $\partial\Omega$ is uniformly convex and $C^2$, 
there exist constants $\overline{m_0}$ and $\overline{M_0}$ satisfying  $0<\overline{m_0}\leq\overline{M_0}$ such that
\[\overline{m_0}\leq\kappa_i(y_0)\leq\overline{M_0},\ \ \ \forall i=1,\cdots,n-1. \]
Now we take $\mu_0'=\min\{\mu_0,\frac{1}{2\overline{M_0}}\}$ and still denote it as $\mu_0$. Then for any $x_0\in\Gamma_{\mu_0}$, we have
\[1-\kappa_i(y_0)\cdot d(x_0)\geq 1-\overline{M_0}\cdot \mu_0>\frac{1}{2}>0, 
\]
and hence 
\[\frac{\kappa_i(y_0)}{1-\kappa_i(y_0)\cdot d(x_0)} \geq 
\overline{m_0}>0. 
\]
It follows that 
\begin{align*}
		-D^2d(x_0)\geq 	\operatorname{diag}	\left(\overline{m_0},\cdots,	\overline{m_0} , 0\right),\ \ \ \ \forall x_0\in\Gamma_{\mu_0}.
\end{align*}

Based on the observation 
\[4\operatorname{diam}(\Omega)-2d(x_0)\in [2\operatorname{diam}(\Omega),4\operatorname{diam}(\Omega)],\]
we further derive 
\begin{align*}
&\ \ \ \ \left(4\operatorname{diam}(\Omega)-2d(x_0)\right)\cdot \left(-D^2d(x_0)\right)+2Dd(x_0)\otimes Dd(x_0)\\
&\geq \operatorname{diag}	\left(2\operatorname{diam}(\Omega)\cdot\overline{m_0},\cdots,	2\operatorname{diam}(\Omega)\cdot\overline{m_0} , 0\right)+\operatorname{diag}	\left(0,\cdots,	0, 2\right)\\
&=\operatorname{diag}	\left(2\operatorname{diam}(\Omega)\cdot\overline{m_0},\cdots,	2\operatorname{diam}(\Omega)\cdot\overline{m_0} , 2\right),\ \ \ \ \forall x_0\in\Gamma_{\mu_0}.
\end{align*}
Up to now, we can deduce  
\begin{align*}
	D^2\left( -4\operatorname{diam}(\Omega)\cdot d(x)+(d(x))^2\right)\big|_{x=x_0}
	\geq \operatorname{diag}	\left(2\operatorname{diam}(\Omega)\cdot\overline{m_0},\cdots,	2\operatorname{diam}(\Omega)\cdot\overline{m_0} , 2\right),\ \  \ \forall x_0\in\Gamma_{\mu_0},
\end{align*}
which means that  $-4\operatorname{diam}(\Omega)\cdot d(x)+(d(x))^2$ is uniformly convex over $\Gamma_{\mu_0}$.  This concludes the proof of the lemma. 
\end{proof}

\begin{Lemma}[Extension function  $\phi_0$ over $\overline{\Omega}$]\label{lemma:uniform-convex}
	If $\Omega$ is a uniformly convex domain of class $C^2$, then there exists a function $\phi_0$ such that
	\begin{enumerate}[(i)]
		\item  $\phi_0\in C^2(\overline{\Omega})$;
		\item $\phi_0$ is uniformly convex over $\overline{\Omega}$;
		\item $\phi_0\big|_{\partial\Omega}=0$.
	\end{enumerate}
\end{Lemma}

 \begin{proof}
 	The function $-4\operatorname{diam}(\Omega)\cdot d(x)+(d(x))^2$ introduced in Lemma \ref{lemma:inspiration} can be extended from the near-boundary region $\Gamma_{\mu_0}$ to the interior region  $\overline{\Omega}\setminus\Gamma_{\mu_0}$ in such a way that the resulting function, denoted as $\phi_0$, is $C^2$ and uniformly convex throughout $\overline{\Omega}$.  Now $\phi_0$ serves as the desired function for this lemma.
 \end{proof}

By Lemma \ref{lemma:uniform-convex}, we have 
\begin{equation}\label{phi-0-1}
	\phi_0(x)<0,\ \ \forall x\in\Omega.
\end{equation}
Denote 
\begin{equation}\label{phi0-max}
	\eta_0:=\max_{x\in\overline{\Omega}}|\phi_0(x)|.
\end{equation}
Hence  $\eta_0>0$, and there always exists some point $x_0\in\Omega$ such that
\begin{equation}\label{phi-0-2}
	|\phi_0(x_0)|=\eta_0.
\end{equation} 
Now we are ready to derive the distance function estimate in the following lemma. For the $n=1$ case, this estimate can be easily depicted as   Figure \ref{fig-lemma}.

\begin{figure}[htbp]  
	\centering 
	\includegraphics[scale=0.32]{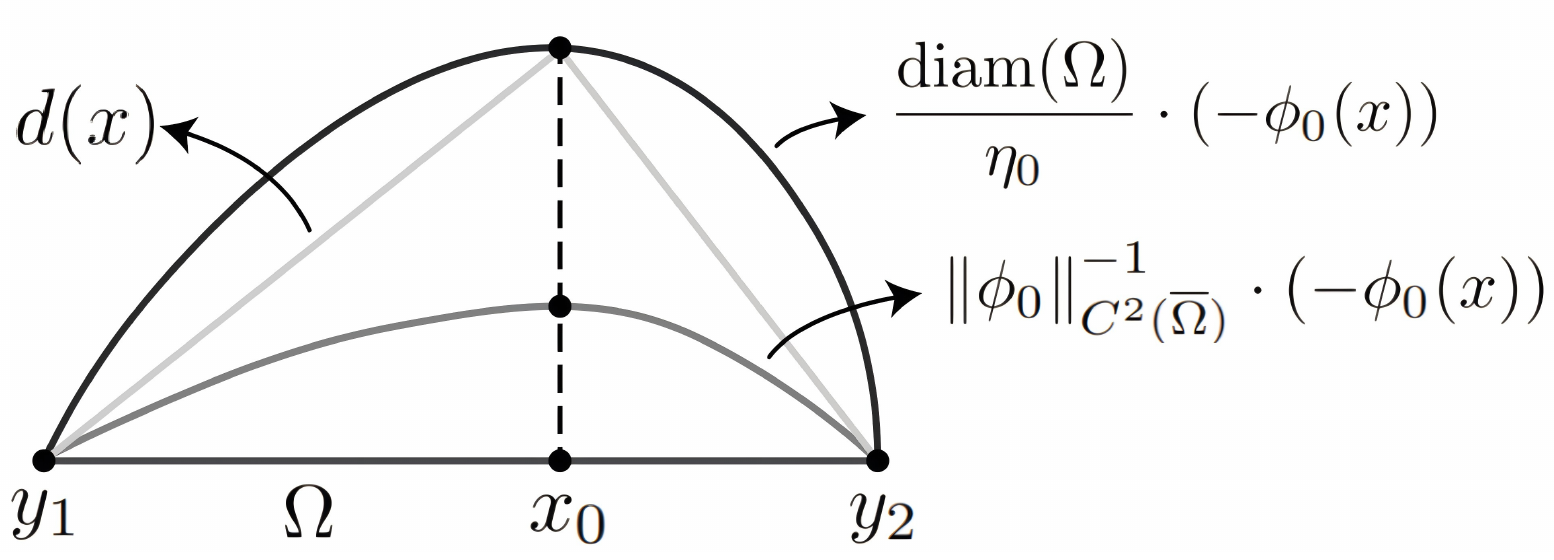} 
	\caption{Estimate of distance function $d(x)$.	} 
	\label{fig-lemma} 
\end{figure}

\begin{Lemma}[Estimate of distance function $d(x)$]\label{lemma:dx}
	There holds
	\begin{equation*}
		\left\|\phi_0\right\|^{-1}_{C^2(\overline{\Omega})}\cdot\left(-\phi_0(x)\right)\leq d(x)\leq \frac{\operatorname{diam}(\Omega)}{\eta_0}\cdot\left(-\phi_0(x)\right),\ \ \forall x\in\Omega.
	\end{equation*}
\end{Lemma}

\begin{proof}
	We note that for any $x\in\Omega$, there exists $z\in\partial\Omega$ such that $d(x)=|x-z|$. Since $\phi_0$ is a convex function and $\phi_0(z)=0$, we can derive
	\begin{equation}\label{inequality-1}
		\frac{|\phi_0(x)|}{d(x)}=\frac{|\phi_0(x)-\phi_0(z)|}{|x-z|}\leq \|D\phi_0\|_{L^\infty(\Omega)}\leq \|\phi_0\|_{C^2(\overline{\Omega})}.
	\end{equation}
	
	Fix any $x\in\Omega$ with $x\neq x_0$. We use $L$ to denote the line connecting $x_0$ and $x$. Due to the convexity of $\Omega$, there always exist two intersection points of $L$ and $\partial\Omega$. Without loss of generality, we can assume that $L\cap \partial\Omega=\{y_1,y_2\}$ and $x$ is between $x_0$ and $y_1$. Since $\phi_0$ restricted to $L$ is also a convex function, we have 
	\begin{equation*}
		\frac{|\phi_0(y_1)-\phi_0(x)|}{|y_1-x|}\geq\frac{|\phi_0(y_1)-\phi_0(x_0)|}{|y_1-x_0|}.
	\end{equation*}
Using  $\phi_0(y_1)=0$ and \eqref{phi-0-2}, we further obtain 
	\begin{equation}\label{inequality-2}
		|\phi_0(x)|\geq\frac{|y_1-x|}{|y_1-x_0|}|\phi_0(x_0)|\geq\frac{d(x)}{\operatorname{diam}(\Omega)}\eta_0.
	\end{equation}
	
	Combining \eqref{inequality-1} and \eqref{inequality-2} gives rise to 
	\begin{equation*}
		\left\|\phi_0\right\|^{-1}_{C^2(\overline{\Omega})}\left|\phi_0(x)\right|\leq d(x)\leq \frac{\operatorname{diam}(\Omega)}{\eta_0}\left|\phi_0(x)\right|,\ \ \forall x\in\Omega.
	\end{equation*}
	Together with \eqref{phi-0-1}, this leads us to the desired result.
\end{proof}

\subsection{Eigenvalue estimates}

Different from \cite{Li-Li-2022}, we will employ quadratic form   to do eigenvalue estimates for the Hessian matrix  of convex function   in this paper. 

For convenience, we provide general expressions in quadratic form for the minimum and maximum eigenvalues of a positive semi-definite matrix in the following lemma. The result is evident, so we will omit its proof.
\begin{Lemma}[Eigenvalue estimates in quadratic form]\label{lemma:eigenvalue} If $A$ is a positive semi-definite matrix, then we have
	\begin{align*}
		\lambda_{\min}(A)&=\min_{\bm{\xi}\neq0}\,\frac{\bm{\xi}^TA\,\bm{\xi}}{\bm{\xi}^T\bm{\xi}}=\min_{|\bm{\xi}|=1}\,\bm{\xi}^TA\,\bm{\xi},\\
		\lambda_{\max}(A)&=\max_{\bm{\xi}\neq0}\,\frac{\bm{\xi}^TA\,\bm{\xi}}{\bm{\xi}^T\bm{\xi}}=\max_{|\bm{\xi}|=1}\,\bm{\xi}^TA\,\bm{\xi}.
	\end{align*}
	Here,  $\frac{\bm{\xi}^TA\,\bm{\xi}}{\bm{\xi}^T\bm{\xi}}$ is known as  
	the Rayleigh quotient.  
\end{Lemma}

Based on Lemma \ref{lemma:eigenvalue}, we have the following two remarks for $\phi_0$ introduced in Lemma \ref{lemma:uniform-convex}:

\begin{Remark}\label{remark:rho}
	Since $\phi_0\in C^2(\overline{\Omega})$ is uniformly convex, then 
	\begin{equation*}
		\text{$D^2\phi_0$ is positive definite, i.e. $D^2\phi_0>0$},
	\end{equation*}
	and we can denote 
	\begin{align*}
		\underline{\rho_0}:&=\min_{x\in\overline{\Omega},\,|\bm{\xi}|=1}\,\bm{\xi}^TD^2\phi_0(x) \,\bm{\xi},\\
		\overline{\rho_0}:&=\max_{x\in\overline{\Omega},\,|\bm{\xi}|=1}\,\bm{\xi}^TD^2\phi_0(x)\,\bm{\xi},
	\end{align*}
	such that 
	\begin{equation*} 
		\underline{\rho_0}\,\mathrm{I}_n\leq D^2\phi_0\leq  \overline{\rho_0}\,\mathrm{I}_n,
	\end{equation*}
	where $\mathrm{I}_n$ is the $n$-order unit matrix, and the constants $\underline{\rho_0}$ and $\overline{\rho_0}$  satisfy $0<\underline{\rho_0}\leq \overline{\rho_0}$.
\end{Remark}

\begin{Remark}\label{remark:product}
	We note that all the $n$ eigenvalues of matrix  $D\phi_0\otimes D\phi_0$ are $0,\cdots,0,|D\phi_0|^2$. It follows that
	\begin{equation*}
		\text{$D\phi_0\otimes D\phi_0$ is  positive semi-definite, i.e. $D\phi_0\otimes D\phi_0\geq 0$.}
	\end{equation*}
	Moreover, if we use $\bm{\xi_1}$ and $\bm{\xi_2}$ to denote the 
	unit eigenvectors of the matrix $D\phi_0\otimes D\phi_0$  corresponding to the eigenvalues $0$ and $|D\phi_0|^2$ respectively, then we have 
	\begin{align*}
		\min_{|\bm{\xi}|=1}\,\bm{\xi}^TD\phi_0\otimes D\phi_0\,\bm{\xi}&=\bm{\xi_1}^TD\phi_0\otimes D\phi_0\,\bm{\xi_1}= 0,\\
		\max_{|\bm{\xi}|=1}\,\bm{\xi}^TD\phi_0\otimes D\phi_0\,\bm{\xi}&=\bm{\xi_2}^TD\phi_0\otimes D\phi_0\,\bm{\xi_2}= |D\phi_0|^2.
	\end{align*}
\end{Remark}

\section{Passing from classical sub-solution  to viscosity solution}
 \label{Sec:lemma2}
 
This section is dedicated to proving a key lemma that facilitates the passage from classical sub-solutions to viscosity solutions for the problem \eqref{equation}-\eqref{bcondition}. 
 
 \begin{Lemma}[Key Lemma]\label{lemma2}
 If the equation \eqref{equation} admits a convex viscosity sub-solution $v\in C(\overline{\Omega})$ satisfying $v\big|_{\partial\Omega}= \varphi$, then it  admits a convex viscosity solution $u\in C(\overline{\Omega})$ satisfying $u\big|_{\partial\Omega}= \varphi$. In particular, if  the equation \eqref{equation} admits a convex classical sub-solution  $v\in C(\overline{\Omega})$ satisfying $v\big|_{\partial\Omega}= \varphi$, then it  admits a convex viscosity solution $u\in C(\overline{\Omega})$ satisfying $u\big|_{\partial\Omega}= \varphi$. 
 \end{Lemma}

The proof of Lemma \ref{lemma2} is based on an adaption of the Perron method. 
We first define
\begin{equation}\label{defineS}
S:=\{v:\ v\in C(\overline{\Omega}),\ v\big|_{\partial\Omega}= \varphi,\ \text{ and }
v \text{ is a convex viscosity sub-solution to \eqref{equation}  over $\Omega$}\}.
\end{equation}
Since the condition of this lemma already gives $S\neq \emptyset$, 
our goal is reduced to showing that  
\begin{equation}\label{defineu}
	u_0(x):=\sup_{v\in S}\,v(x),\ \ \forall x\in\overline{\Omega}
\end{equation}
is the desired convex viscosity solution to \eqref{equation}.  
The rest of the proof can be split into five steps. 

\subsection*{Step 1:} \textit{$u_0\big|_{\partial\Omega}= \varphi$.} 

This is clear because it holds  for any $x\in\partial\Omega$ that 
\begin{equation*}
	u_0(x)=\sup_{v\in S}\,v(x)=\sup_{v\in S}\,\varphi(x)=\varphi(x).
\end{equation*}

\subsection*{Step 2:} \textit{$u_0$ is a convex function over $\overline{\Omega}$.} 

For any $x_0,y_0\in\overline{\Omega}$, $t_0\in(0,1)$, and $z_0=t_0x_0+(1-t_0)y_0\in\Omega$, we can infer from the definition \eqref{defineu} that there exists a sequence $\{v_k\}\subset S$ such that $$\lim_{k\to\infty}v_k(z_0)=u_0(z_0).$$ Since $v_k$ is convex, there holds
$$v_k(z_0)\leq t_0v_k(x_0)+(1-t_0)v_k(y_0).$$  It follows that 
\begin{align*}
	u_0(z_0)&=\lim_{k\to\infty}v_k(z_0)=\varlimsup_{k\to\infty}v_k(z_0)\\
	&\leq \varlimsup_{k\to\infty}\left(t_0v_k(x_0)+(1-t_0)v_k(y_0)\right)\\
	&\leq 
	t_0u_0(x_0)+(1-t_0)u_0(y_0).
\end{align*}
Thus $u_0$ is a convex function over $\overline{\Omega}$. 

\subsection*{Step 3:}
\textit{$u_0\in C(\overline{\Omega})$.} 

From Step 1, we have $u_0\big|_{\partial\Omega}= \varphi$. 
Using Step 2 and the openness of $\Omega$, we note that $u_0\in C(\Omega)$. It remains to prove that $u_0$ is also continuous on $\partial\Omega$. 

For any $z_0\in\Omega$, let $\mathcal{C}_{z_0}(x)$ be the cone  function generated by the point $(z_0,u_0(z_0))$ and the set of all boundary value points  $\{(x,\varphi(x)):\, x\in\partial\Omega\}$. By virtue of Step 2, we have 
$\mathcal{C}_{z_0}(x)\geq u_0(x)$ for any  $x\in\overline{\Omega}$. 
Clearly, we also have $\mathcal{C}_{z_0}(x)\in C(\overline{\Omega})$ and $\mathcal{C}_{z_0}(x)\big|_{\partial\Omega}= \varphi$. Thus we can derive
\begin{equation}\label{eqI}
\lim_{x\to y\in\partial\Omega}\mathcal{C}_{z_0}(x)=\varphi(y).
\end{equation}

Now we arbitrarily take $v_0\in S$. Using the definition \eqref{defineS} of $S$ then gives rise to $v_0(x)\in C(\overline{\Omega})$ and $v_0(x)\big|_{\partial\Omega}= \varphi$. Hence we obtain 
\begin{equation}\label{eqII}
	\lim_{x\to y\in\partial\Omega}v_0(x)=\varphi(y).
\end{equation}

By the definition \eqref{defineu} of $u_0$, we note that $	u_0(x)\geq v_0(x)$ for any  $x\in\overline{\Omega}$. It follows that 
\begin{equation}\label{eqIII}
	\mathcal{C}_{z_0}(x)\geq u_0(x)\geq v_0(x),\ \ \forall x\in\overline{\Omega},
\end{equation}
\begin{figure}[htbp]  
	\centering 
	\includegraphics[scale=0.22]{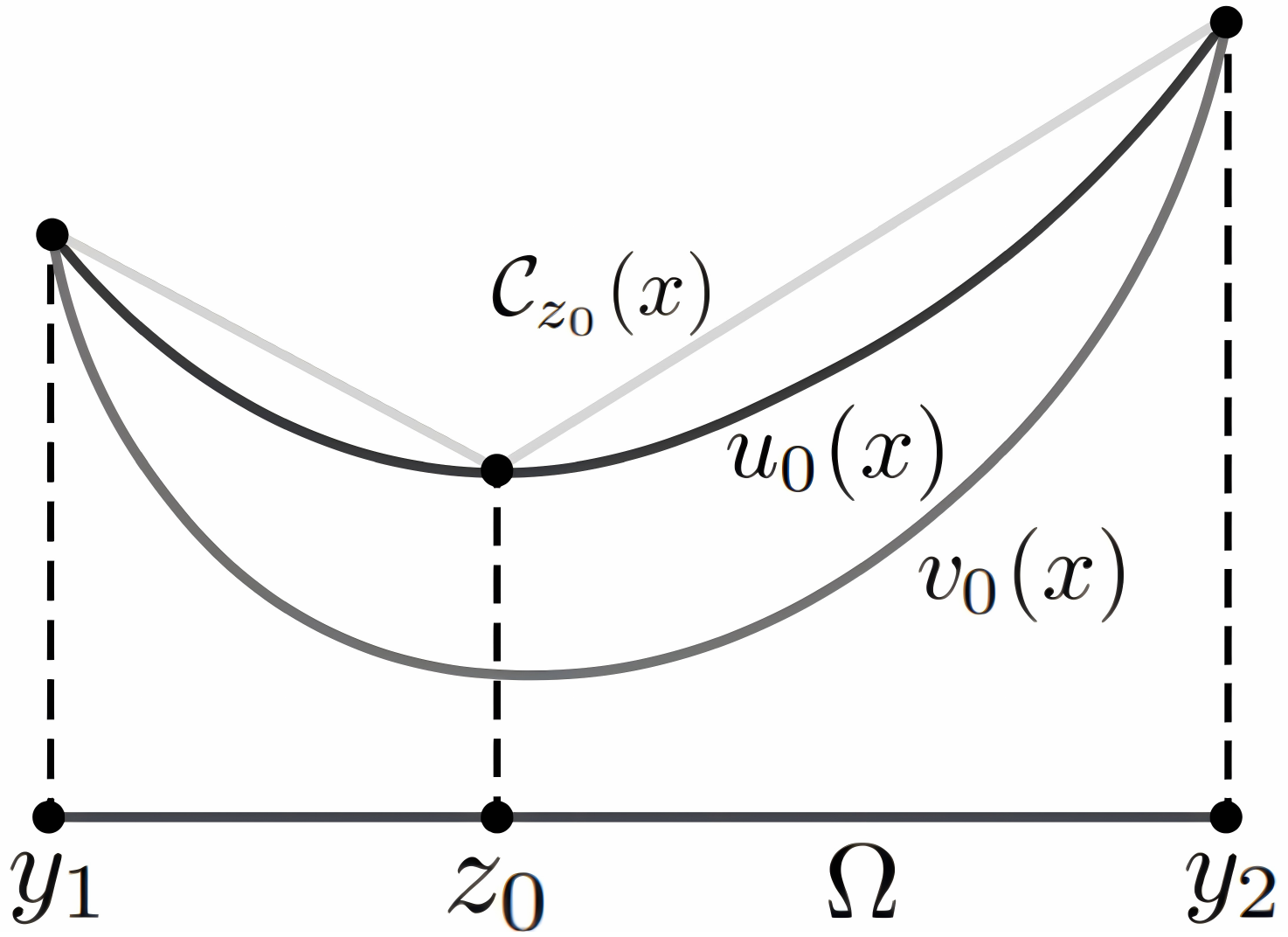} 
	\caption{Relations among the functions $\mathcal{C}_{z_0}(x)$, $u_0(x)$, and $v_0(x)$.	} 
	\label{fig-step} 
\end{figure}
For the $n=1$ case, this relation can be easily illustrated by Figure \ref{fig-step}, where $y_1,y_2\in\partial\Omega$, both $\mathcal{C}_{z_0}(x)$ and $u_0(x)$ take the minimum at $z_0$.  
Combining \eqref{eqI}, \eqref{eqII} and \eqref{eqIII} together, we can apply the squeeze theorem to get
\begin{equation*} 
	\lim_{x\to y\in\partial\Omega}u_0(x)=\varphi(y),
\end{equation*}
which implies that $u_0$ is continuous on $\partial\Omega$. Up to now, we have proved that $u_0$ is  continuous over $\overline{\Omega}$.

\subsection*{Step 4:} 
\textit{$u_0$ is a viscosity sub-solution to \eqref{equation}  over $\Omega$.}

The proof of this step is by contradiction. We suppose $u_0$ does not satisfy the condition of viscosity sub-solution at some point $\overline{x_0}\in\Omega$. Then there exists  an open neighborhood $U_0(\overline{x_0})\subset\Omega$ and a convex function $\overline{\phi_0}\in C^2(U_0(\overline{x_0}))$ 
satisfying
\begin{equation}\label{eq:minimum}
	(\overline{\phi_0}-u_0)(x)\geq (\overline{\phi_0}-u_0)(\overline{x_0}),\ \ \forall x\in U_0(\overline{x_0})
\end{equation}
such that
\begin{equation*}
	F\left(\lambda_1(D^2 \overline{\phi_0}(\overline{x_0})),\cdots,\lambda_n(D^2 \overline{\phi_0}(\overline{x_0}))\right)< g(\overline{x_0})(d(\overline{x_0}))^{\alpha}.
\end{equation*}

Denote 
\begin{equation*}
	\overline{\zeta_0}:=g(\overline{x_0})(d(\overline{x_0}))^{\alpha}-F\left(\lambda_1(D^2 \overline{\phi_0}(\overline{x_0})),\cdots,\lambda_n(D^2 \overline{\phi_0}(\overline{x_0}))\right).
\end{equation*}
It follows that $\overline{\zeta_0}>0$ and 
\begin{equation*}
	F\left(\lambda_1(D^2 \overline{\phi_0}(\overline{x_0})),\cdots,\lambda_n(D^2 \overline{\phi_0}(\overline{x_0}))\right)< g(\overline{x_0})(d(\overline{x_0}))^{\alpha}-\frac{\overline{\zeta_0}}{2}.
\end{equation*}
Since $F\in C(\mathbb{R}^n)$,   $g(x)\in C(\overline{\Omega})$, $d(x)\in C(\overline{\Omega})$ and $\overline{\phi_0}\in C^2(U_0(\overline{x_0}))$, then there exists $\overline{\delta_0}>0$ such that for any $x\in B(\overline{x_0},\overline{\delta_0})\subset U_0(\overline{x_0})$, we have 
\begin{equation*}
	F\left(\lambda_1(D^2 \overline{\phi_0}(x)),\cdots,\lambda_n(D^2 \overline{\phi_0}(x))\right)< g(x)(d(x))^{\alpha}-\frac{\overline{\zeta_0}}{4},\ \ \forall x\in B(\overline{x_0},\overline{\delta_0}).
\end{equation*}

Let
\begin{equation}\label{eq:overline-phi}
	\overline{\phi_1}(x):=\overline{\phi_0}(x)+\frac{\overline{\varepsilon}|x-\overline{x_0}|^2}{2}-\left(\overline{\phi_0}-u_0\right)(\overline{x_0})-\frac{\overline{\varepsilon}\overline{\delta_0}^2}{8},\ \ \forall x\in B(\overline{x_0},\overline{\delta_0}),
\end{equation}
where $\overline{\varepsilon}\in (0,1]$ is a constant to be determined later. We notice that 
\[D^2\overline{\phi_1}(x)=D^2\overline{\phi_0}(x)+\overline{\varepsilon}\, \mathrm{I}_n,\ \ \ \ \lambda_k\left(D^2\overline{\phi_1}(x)\right)=\lambda_k\left(D^2\overline{\phi_0}(x)\right)+\overline{\varepsilon},\ \ \forall k=1,\cdots,n.\]

Denote   $$R_0:=\left\|\overline{\phi_0}(x)\right\|_{C^2(\overline{B(\overline{x_0},\overline{\delta_0})})}.$$ Then we have
\begin{align*}
	&	\left(\lambda_1(D^2 \overline{\phi_1}(x)),\cdots,\lambda_n(D^2 \overline{\phi_1}(x))\right)\in [0,R_0+1]^n,\ \ \forall x\in B(\overline{x_0},\overline{\delta_0}),\\ 
	& 	\left(\lambda_1(D^2 \overline{\phi_0}(x)),\cdots,\lambda_n(D^2 \overline{\phi_0}(x))\right)\in [0,R_0+1]^n,\ \ \forall x\in B(\overline{x_0},\overline{\delta_0}).
\end{align*}
By virtue of $F\in C(\mathbb{R}^n)$,  it is clear that $F$ is uniformly continuous on $[0,R_0+1]^n$, which means there exists   a   constant $\overline{\varepsilon_0}>0$ such that for any $0<\overline{\varepsilon}\leq \overline{\varepsilon_0}$, we have 
\begin{equation*}
	\left|F\left(\lambda_1(D^2 \overline{\phi_1}(x)),\cdots,\lambda_n(D^2 \overline{\phi_1}(x))\right)-	F\left(\lambda_1(D^2 \overline{\phi_0}(x)),\cdots,\lambda_n(D^2 \overline{\phi_0}(x))\right)\right|<\frac{\overline{\zeta_0}}{8},\ \ \forall x\in B(\overline{x_0},\overline{\delta_0}).
\end{equation*}
Up to now, we get
\begin{equation}\label{eq:Q1}
	F\left(\lambda_1(D^2 \overline{\phi_1}(x)),\cdots,\lambda_n(D^2 \overline{\phi_1}(x))\right)< g(x)(d(x))^{\alpha}-\frac{\overline{\zeta_0}}{8},\ \ \forall x\in B(\overline{x_0},\overline{\delta_0}).
\end{equation}

Without loss of generality, we now take $\overline{\varepsilon}=\overline{\varepsilon_0}$. 
Consider 
\begin{equation}\label{defineU1}
	U_1:=\{x\in B(\overline{x_0},\overline{\delta_0}):\ \overline{\phi_1}(x)<u_0(x)\}.
\end{equation}
Hence, for any $x\in U_1$, using definition  \eqref{eq:overline-phi} now leads us to 
\begin{equation*}
	\overline{\phi_0}(x)+\frac{\overline{\varepsilon_0}|x-\overline{x_0}|^2}{2}-\left(\overline{\phi_0}-u_0\right)(\overline{x_0})-\frac{\overline{\varepsilon_0}\overline{\delta_0}^2}{8}<u_0(x),
\end{equation*}
that is, 
\begin{equation*}
	\frac{\overline{\varepsilon_0}|x-\overline{x_0}|^2}{2}<\left(\overline{\phi_0}-u_0\right)(\overline{x_0})-\left(\overline{\phi_0}-u_0\right)(x)+\frac{\overline{\varepsilon_0}\overline{\delta_0}^2}{8}.
\end{equation*}
which along with \eqref{eq:minimum} further gives
\begin{equation*}
	|x-\overline{x_0}| < \frac{\overline{\delta_0}}{2},\ \ \  \ \text{i.e. }	x\in B(\overline{x_0},\frac{\overline{\delta_0}}{2}).
\end{equation*}
Now we obtain
\begin{equation}\label{relationU1}
	U_1\subset B(\overline{x_0},\frac{\overline{\delta_0}}{2})\subset\subset B(\overline{x_0},\overline{\delta_0}).
\end{equation}

In view of \eqref{defineu} and \eqref{eq:overline-phi}, we note that
\begin{align*}
	u_0(\overline{x_0})&=\sup_{v\in S}\,v(\overline{x_0}),\\
	\overline{\phi_1}(\overline{x_0})&=u_0(\overline{x_0})-\frac{\overline{\varepsilon_0}\overline{\delta_0}^2}{8}.
\end{align*}
Then by deﬁnition of supremum, there exists $v_1\in S$ such that 
\begin{equation*}
	v_1(\overline{x_0})>u_0(\overline{x_0})-\frac{\overline{\varepsilon_0}\overline{\delta_0}^2}{16}>u_0(\overline{x_0})-\frac{\overline{\varepsilon_0}\overline{\delta_0}^2}{8}=\overline{\phi_1}(\overline{x_0}).
\end{equation*}
Let
\[U_2:=\{x\in B(\overline{x_0},\overline{\delta_0}):\ \overline{\phi_1}(x)<v_1(x)\}.\]
Since $v_1(x)\leq u_0(x)$ holds for any $x\in\Omega$, then by \eqref{defineU1} and \eqref{relationU1}, we have 
\begin{equation*} 
	U_2\subset	U_1\subset\subset B(\overline{x_0},\overline{\delta_0})\subset\Omega.
\end{equation*}

We turn to consider $\overline{\phi_1}-v_1$ on $U_2$. By virtue of $\overline{\phi_1}-v_1\in C(B(\overline{x_0},\overline{\delta_0}))$,  there always exists some point $z_0\in\overline{U_2}$ such that 
\begin{equation*}
	\left(\overline{\phi_1}-v_1\right)(z_0)=\min_{x\in\overline{U_2}}\left(\overline{\phi_1}-v_1\right)(x).
\end{equation*}
Together with  $\left(\overline{\phi_1}-v_1\right)\big|_{\partial U_2}=0$, we can derive $z_0\in U_2$, which implies that $\overline{\phi_1}-v_1$ indeed takes its minimum at $z_0$ in the interior of $U_2$. Due to $\overline{\phi_0}\in C^2(U_0(\overline{x_0}))$ and \eqref{eq:overline-phi}, we have $\overline{\phi_1}\in C^2(U_2)$. 
We also note that $v_1\in S$ is a viscosity sub-solution to \eqref{equation} over $\Omega$. Then Definition \ref{def:vis} leads us to 
\begin{equation*}
	F\left(\lambda_1(D^2 \overline{\phi_1}(z_0)),\cdots,\lambda_n(D^2 \overline{\phi_1}(z_0))\right)\geq  g(z_0)(d(z_0))^{\alpha},
\end{equation*}
which contradicts \eqref{eq:Q1}. Thus we are ready to conclude that $u_0$ is a viscosity sub-solution to \eqref{equation}  over $\Omega$.

\subsection*{Step 5:}
\textit{$u_0$ is a viscosity super-solution to \eqref{equation}  over $\Omega$.}

By the above four steps, we already obtain $u_0\in S$. 
The proof of this step is also by contradiction. We suppose $u_0$ does not satisfy the condition of viscosity super-solution at some point $\underline{x_0}\in\Omega$. Then there exists  an open neighborhood $V_0(\underline{x_0})\subset\Omega$ and a convex function $\underline{\phi_0}\in C^2(V_0(\underline{x_0}))$ 
satisfying
\begin{equation*} 
	(\underline{\phi_0}-u_0)(x)\leq (\underline{\phi_0}-u_0)(\underline{x_0}),\ \ \forall x\in V_0(\underline{x_0})
\end{equation*}
such that
\begin{equation*}
	F\left(\lambda_1(D^2 \underline{\phi_0}(\underline{x_0})),\cdots,\lambda_n(D^2 \underline{\phi_0}(\underline{x_0}))\right)> g(\underline{x_0})(d(\underline{x_0}))^{\alpha}.
\end{equation*}

Denote 
\begin{equation*}
	\underline{\zeta_0}:=F\left(\lambda_1(D^2 \underline{\phi_0}(\underline{x_0})),\cdots,\lambda_n(D^2 \underline{\phi_0}(\underline{x_0}))\right)-g(\underline{x_0})(d(\underline{x_0}))^{\alpha}.
\end{equation*}
Then we have $\underline{\zeta_0}>0$ and 
\begin{equation*}
	F\left(\lambda_1(D^2 \underline{\phi_0}(\underline{x_0})),\cdots,\lambda_n(D^2 \underline{\phi_0}(\underline{x_0}))\right)> g(\underline{x_0})(d(\underline{x_0}))^{\alpha}+\frac{\underline{\zeta_0}}{2}.
\end{equation*}
Since $F\in C(\mathbb{R}^n)$,   $g(x)\in C(\overline{\Omega})$, $d(x)\in C(\overline{\Omega})$ and $\underline{\phi_0}\in C^2(V_0(\underline{x_0}))$, then there exists $\underline{\delta_0}>0$ such that for any $x\in B(\underline{x_0},\underline{\delta_0})\subset V_0(\underline{x_0})$, we have 
\begin{equation*}
	F\left(\lambda_1(D^2 \underline{\phi_0}(x)),\cdots,\lambda_n(D^2 \underline{\phi_0}(x))\right)> g(x)(d(x))^{\alpha}+\frac{\underline{\zeta_0}}{4},\ \ \forall x\in B(\underline{x_0},\underline{\delta_0}).
\end{equation*}

Let
\begin{equation}\label{eq:underline-phi}
	\underline{\phi_1}(x):=\underline{\phi_0}(x)-\frac{\underline{\varepsilon}|x-\underline{x_0}|^2}{2}-\left(\underline{\phi_0}-u_0\right)(\underline{x_0})+\frac{\underline{\varepsilon}\underline{\delta_0}^2}{8},\ \ \forall x\in B(\underline{x_0},\underline{\delta_0}),
\end{equation}
where $\underline{\varepsilon}>0$ is a constant to be determined later. It is clear that 
\[D^2\underline{\phi_1}(x)=D^2\underline{\phi_0}(x)-\underline{\varepsilon}\, \mathrm{I}_n,\ \ \ \ \lambda_k\left(D^2\underline{\phi_1}(x)\right)=\lambda_k\left(D^2\underline{\phi_0}(x)\right)-\underline{\varepsilon},\ \ \forall k=1,\cdots,n.\]
Similar to the arguments used in Step 4, we can derive that there exists   a   constant $\underline{\varepsilon_0}>0$ such that for any $0<\underline{\varepsilon}\leq \underline{\varepsilon_0}$, we have 
\begin{equation*}
	\left|F\left(\lambda_1(D^2 \underline{\phi_1}(x)),\cdots,\lambda_n(D^2 \underline{\phi_1}(x))\right)-	F\left(\lambda_1(D^2 \underline{\phi_0}(x)),\cdots,\lambda_n(D^2 \underline{\phi_0}(x))\right)\right|<\frac{\underline{\zeta_0}}{8},\ \ \forall x\in B(\underline{x_0},\underline{\delta_0}).
\end{equation*}
Hence we obtain
\begin{equation}\label{eq:Q1'}
	F\left(\lambda_1(D^2 \underline{\phi_1}(x)),\cdots,\lambda_n(D^2 \underline{\phi_1}(x))\right)> g(x)(d(x))^{\alpha}+\frac{\underline{\zeta_0}}{8},\ \ \forall x\in B(\underline{x_0},\underline{\delta_0}).
\end{equation} 
Without loss of generality, we now take $\underline{\varepsilon}=\underline{\varepsilon_0}$. 
Consider 
\begin{equation}\label{defineV1}
	V_1:=\{x\in B(\underline{x_0},\underline{\delta_0}):\ \underline{\phi_1}(x)>u_0(x)\}.
\end{equation}
Then for any $x\in V_1$, by using the similar arguments in the Step 4, we can infer that 
\begin{equation}\label{relationV1}
	V_1\subset B(\underline{x_0},\frac{\underline{\delta_0}}{2})\subset\subset B(\underline{x_0},\underline{\delta_0})\subset\Omega.
\end{equation}

Define
\begin{equation}\label{eq:underline-u0}
	\underline{u_0}(x):=\begin{cases}
		\underline{\phi_1}(x)\ \ &\text{ if }x\in V_1,\\
		u_0(x),\ \ &\text{ if }x\in\overline{\Omega}\setminus V_1.
	\end{cases}
\end{equation}
In fact, by \eqref{defineV1} and \eqref{relationV1}, we have
\begin{equation*}
	\underline{u_0}(x)=\begin{cases}
		\max\,\{u_0(x),\underline{\phi_1}(x)\}\ \ &\text{ if }x\in B(\underline{x_0},\underline{\delta_0}),\\
		u_0(x),\ \ &\text{ if }x\in\overline{\Omega}\setminus B(\underline{x_0},\underline{\delta_0}).
	\end{cases}
\end{equation*}
We also note that 
\begin{equation*}
	\max\,\{u_0(x),\underline{\phi_1}(x)\}=u_0(x), \ \ \ \ \forall x\in \partial B(\underline{x_0},\underline{\delta_0}).
\end{equation*}
Then it follows that $\underline{u_0}\in C(\overline{\Omega})$.  
We have already proved in Step 4 that $u_0$ is a viscosity sub-solution to \eqref{equation}  over $\Omega$. In view of  $\underline{\phi_0}\in C^2(U_0(\underline{x_0}))$ and \eqref{eq:underline-phi}, we have $\underline{\phi_1}\in C^2(B(\underline{x_0},\underline{\delta_0}))$. 
By virtue of \eqref{eq:Q1'} and Definition \ref{def:class}, $\underline{\phi_1}$ is a classical sub-solution to \eqref{equation}  over $B(\underline{x_0},\underline{\delta_0})$, and then by Lemma \ref{lemma:equiv}, $\underline{\phi_1}$ is also a viscosity sub-solution to \eqref{equation}  over $B(\underline{x_0},\underline{\delta_0})$. Therefore, $\max\,\{u_0,\underline{\phi_1}\}$ is a viscosity sub-solution to \eqref{equation}  over $B(\underline{x_0},\underline{\delta_0})$, and thus $\underline{u_0}$ is a viscosity sub-solution to \eqref{equation}  over $B(\underline{x_0},\underline{\delta_0})$. 
Using \eqref{eq:underline-u0} and \eqref{bcondition}, we notice that 
\begin{equation*}
\underline{u_0}\big|_{\partial\Omega}=u_0\big|_{\partial\Omega}=\varphi.
\end{equation*}

Thus, by \eqref{defineS}, we can obtain $\underline{u_0}\in S$. 
Consequently, by \eqref{defineu}, we have 
\begin{equation}\label{eq:u0-u0}
	u_0(x)=\sup_{v\in S}\,v(x)\geq \underline{u_0}(x),\ \ \forall x\in\Omega.
\end{equation}
On the other hand,  by \eqref{defineV1} and \eqref{eq:underline-u0}, there holds
\begin{equation*}
\underline{u_0}(x) =
	\underline{\phi_1}(x)>u_0(x),\ \ \forall x\in V_1,
\end{equation*}
which is a contradiction to \eqref{eq:u0-u0} and \eqref{relationV1}. Hence we can summarize that $u_0$  is a viscosity super-solution to \eqref{equation} over $\Omega$.

Up to now, $u_0$ is the viscosity solution $u$ as desired,  and thus we have completed  the proof of Lemma \ref{lemma2}.

\section{Proof  of 
	Existence Theorem}\label{Sec:Case 1-3}

The proof of Theorem \ref{thm1} is divided into four cases according to the value of $\alpha$. For each case, 
we first construct a classically strict  sub-solution $W_k$ by using the eigenvalue estimates in quadratic form;  Lemma \ref{lemma2} is then employed to establish the existence of a convex viscosity solution $u$, while the Comparison Principle (Lemma \ref{lemma1}) ensures that $u\geq W_k$; finally, we apply Lemma \ref{lemma:dx} to obtain the lower bound estimates for $u$ via the distance function $d(x)$.

\subsection{Case 1: $\alpha\in[0,+\infty)$}

Denote 
\begin{equation*} 
W_k=C\phi_0+\psi_k,
\end{equation*}
where $C>0$ is a constant to be determined.  
Then we have
\begin{equation*} 
	D^2	W_k=CD^2\phi_0+D^2\psi_k.
\end{equation*}

By Remark \ref{remark:phi-k}, Lemma \ref{lemma:eigenvalue} and Remark \ref{remark:rho}, we can derive  
\begin{align*}
\lambda_{\min}(D^2W_k)&=\min_{|\bm{\xi}|=1}\,\bm{\xi}^TD^2W_k\,\bm{\xi}=\min_{|\bm{\xi}|=1}\,\left(C\bm{\xi}^TD^2\phi_0\,\bm{\xi}+\bm{\xi}^TD^2\psi_k\,\bm{\xi}\right)\geq C\min_{|\bm{\xi}|=1}\,\bm{\xi}^TD^2\phi_0\,\bm{\xi}\geq C\underline{\rho_0},\\
\lambda_{\max}(D^2W_k)&=\max_{|\bm{\xi}|=1}\,\bm{\xi}^TD^2W_k\,\bm{\xi}=\max_{|\bm{\xi}|=1}\,\left(C\bm{\xi}^TD^2\phi_0\,\bm{\xi}+\bm{\xi}^TD^2\psi_k\,\bm{\xi}\right)\geq C\max_{|\bm{\xi}|=1}\,\bm{\xi}^TD^2\phi_0\,\bm{\xi}\geq  C\underline{\rho_0},
\end{align*}

For any $x\in\Omega$, using the assumptions  $(\mathbf{A_2})$-$(\mathrm{i})$ and $(\mathbf{A_3})$ together with the fact $d(x)\leq \operatorname{diam}(\Omega)$, we can infer 
\begin{align*}
&\ \ \ \ F\left(\lambda_1(D^2 W_k),\cdots,\lambda_n(D^2 W_k)\right)\cdot\left[g(x)(d(x))^{\alpha}\right]^{-1}\\
&\geq \Lambda_0\left(\lambda_{\min}(D^2 W_k)\right)^{a}\left(\lambda_{\max}(D^2 W_k)\right)^{b}\cdot M_0^{-1}(d(x))^{-\alpha}\\
&\geq \Lambda_0\left(C\underline{\rho_0}\right)^{a}\left(C\underline{\rho_0}\right)^{b}\cdot M_0^{-1}(\operatorname{diam}(\Omega))^{-\alpha}\\
&= C^{a+b} \Lambda_0 M_0^{-1} \underline{\rho_0}^{a+b}(\operatorname{diam}(\Omega))^{-\alpha},\ \ \forall x\in\Omega.
\end{align*}
Hence there exists some constant $C_0>0$ such that for any $C\geq C_0$ (without loss of generality, we can take $C=C_0$ here), we have 
\[F\left(\lambda_1(D^2 W_k),\cdots,\lambda_n(D^2 W_k)\right)\cdot\left[g(x)(d(x))^{\alpha}\right]^{-1}>1,\ \ \forall x\in\Omega,\]
i.e. 
\begin{equation}\label{case1-F}
	F\left(\lambda_1(D^2 W_k),\cdots,\lambda_n(D^2 W_k)\right) >g(x)(d(x))^{\alpha},\ \ \forall x\in\Omega.
\end{equation}
Up to now, it is evident that 
\begin{equation}\label{eq:W-case1-new}
	W_k=C_0\phi_0+\psi_k
\end{equation} is a classically strict sub-solution to  \eqref{equation} over $\Omega$.

By Lemma \ref{lemma:uniform-convex}, there holds $\phi_0\big|_{\partial\Omega}=0$. In view of Lemma \ref{lemma:approximation}, we have $\psi_k\big|_{\partial\Omega}=\varphi$. 
Then 
it is  clear from \eqref{eq:W-case1-new}   that $W_k\in C(\overline{\Omega})$ and
\begin{equation*}
	W_k\big|_{\partial\Omega}=\psi_k\big|_{\partial\Omega}=\varphi.
\end{equation*}
 Now, by virtue of \eqref{case1-F} and Lemma \ref{lemma2}, the equation \eqref{equation}    admits a convex viscosity solution $u\in C(\overline{\Omega})$ satisfying $u\big|_{\partial\Omega}= \varphi$, i.e. the problem \eqref{equation}-\eqref{bcondition} admits a viscosity solution $u\in C(\overline{\Omega})$.  Using \eqref{case1-F}  again, we can obtain from the Comparison Principle (Lemma \ref{lemma1}) that  
\begin{equation*} 
	W_k(x)\leq  u(x), \ \ \ \  \forall x\in\overline{\Omega},
\end{equation*}
i.e. 
\begin{equation*}
 u(x)\geq C_0\phi_0(x)+\psi_k(x), \ \ \ \  \forall x\in\overline{\Omega}.
\end{equation*}

Due to Lemma \ref{lemma:dx}, we derive
\begin{equation*}
d(x)\geq \left\|\phi_0\right\|^{-1}_{C^2(\overline{\Omega})}\cdot\left(-\phi_0(x)\right),\ \ \ \  \forall x\in\overline{\Omega},
\end{equation*}
that is, 
\begin{equation*}
	\phi_0(x)\geq \left\|\phi_0\right\|_{C^2(\overline{\Omega})}\cdot\left(-d(x)\right),\ \ \ \  \forall x\in\overline{\Omega}.
\end{equation*}
Hence, we can deduce that 
\begin{equation*}
	u(x)\geq  \psi_k(x)-C_0\left\|\phi_0\right\|_{C^2(\overline{\Omega})} d(x), \ \ \ \  \forall x\in\overline{\Omega}.
\end{equation*}
Denote 
$\overline{C_0}=C_0\left\|\phi_0\right\|_{C^2(\overline{\Omega})} $. Thus, by letting $k\to\infty$, we can conclude that
\begin{equation*}
	u(x)\geq \psi(x)-\overline{C_0}\,d(x), \ \ \forall x\in\overline{\Omega}.
\end{equation*}

\subsection{Case 2: $\alpha\in(-b,0)$}

In this case, we have $-\alpha\in (0,b)$, and then by Lemma \ref{lemma:dx}, we obtain
\begin{equation}\label{dx-case2}
\left(d(x)\right)^{-\alpha}\geq\left\|\phi_0\right\|^{\alpha}_{C^2(\overline{\Omega})}\cdot\left(-\phi_0(x)\right)^{-\alpha},\ \ \ \ \forall x\in\Omega.
\end{equation}

Now we consider
\begin{equation*} 
	W_k=C_1\phi_0+C_2\left(-\phi_0\right)^{1+\theta}+\psi_k,
\end{equation*}
where $C_1>0$, $C_2>0$ and $\theta\in (0,1)$ are constants to be determined. Our first step is to show that there exists $\theta_0\in (0,1)$ such that when $\theta=\theta_0$, if $C_1$ and $C_2$ are sufficiently large, $W_k$ is a classically  strict sub-solution to  \eqref{equation}  over $\Omega$.

 It is clear that 
\begin{align*} 
	DW_k&=C_1D\phi_0-C_2(1+\theta)\left(-\phi_0\right)^{\theta}D\phi_0+D\psi_k,\\
	D^2	W_k&=\left[C_1-C_2(1+\theta)\left(-\phi_0\right)^{\theta}\right]D^2\phi_0+C_2(1+\theta)\theta\left(-\phi_0\right)^{\theta-1}D\phi_0\otimes D\phi_0+D^2\psi_k.
\end{align*}
Without loss of generality, we let 
\begin{equation*}
C_2=\frac{C_1}{2(1+\theta)\eta_0{}^\theta},
\end{equation*} 
where $C_1$ and $\theta$ are still to be determined. By \eqref{phi-0-2}, we have $0<-\phi_0\leq \eta_0$. Now we can derive
\begin{align*}
C_1-C_2(1+\theta)\left(-\phi_0\right)^{\theta}&\geq \frac{C_1}{2},\\
C_2(1+\theta)\theta&= \frac{C_1\theta}{2\eta_0{}^\theta}.
\end{align*}
Hence 
\begin{align*}
D^2	W_k&\geq \frac{C_1}{2}D^2\phi_0+\frac{C_1\theta}{2\eta_0{}^\theta}\left(-\phi_0\right)^{\theta-1}D\phi_0\otimes D\phi_0+D^2\psi_k.
\end{align*}
Similar to the last case, using Lemma \ref{lemma:eigenvalue}, Remark \ref{remark:phi-k},  Remark \ref{remark:rho} and Remark \ref{remark:product}, we can deduce that 
\begin{align*}
	\lambda_{\min}(D^2W_k)&=\min_{|\bm{\xi}|=1}\,\bm{\xi}^TD^2W_k\,\bm{\xi} \geq \frac{C_1}{2}\min_{|\bm{\xi}|=1}\,\bm{\xi}^TD^2\phi_0\,\bm{\xi}\geq\frac{C_1}{2}\underline{\rho_0},\stepcounter{equation}\tag{\theequation}\label{lambda-min}\\
	\lambda_{\max}(D^2W_k)&=\max_{|\bm{\xi}|=1}\,\bm{\xi}^TD^2W_k\,\bm{\xi} \geq \frac{C_1\theta}{2\eta_0{}^\theta}\left(-\phi_0\right)^{\theta-1}\max_{|\bm{\xi}|=1}\,\bm{\xi}^TD\phi_0\otimes D\phi_0\,\bm{\xi}= \frac{C_1\theta}{2\eta_0{}^\theta}\left(-\phi_0\right)^{\theta-1}|D\phi_0|^2.\stepcounter{equation}\tag{\theequation}\label{lambda-max}
\end{align*}

Now we need to  bound $|D\phi_0|$ in the estimate of $\lambda_{\max}(D^2W_k)$ carefully. For this purpose, 
we split $\Omega$ as two disjoint domains:
\begin{equation*}
\Omega:=\Omega_1\cup\Omega_2,
\end{equation*}
where
\begin{equation*}
\Omega_1:=\left\{x\in\Omega:-\frac{\eta_0}{2}\leq\phi_0(x)<0\right\}
\end{equation*}
represents the domain near the boundary,
and
\begin{equation*}
\Omega_2:=\left\{x\in\Omega:-\eta_0\leq\phi_0(x)<-\frac{\eta_0}{2}\right\}
\end{equation*}
represents  the domain away from the boundary.

\subsection*{Case 2-1: In $\Omega_1$, $D\phi_0$ is nondegenerate.} We claim that 
\begin{equation}\label{claim-case21}
|D\phi_0(x)|\geq\frac{\eta_0}{2l_0},\ \ \ \ \forall x\in\Omega_1.
\end{equation}
Now we prove this claim. By virtue of \eqref{phi-0-1} and \eqref{phi-0-2}, we have $\phi_0(x_0)=-\eta_0.$  
For any $x\in\Omega_1$, we let $\bm{e}=\frac{x-x_0}{|x-x_0|}$ and $L_{\bm{e}}=\{y=x+t\bm{e}:\, t\in\mathbb{R}\}$. We notice that $\phi_0$ restricted to the segment $L_{\bm{e}}\cap\partial\Omega$ is also a convex function. Therefore we obtain 
\begin{equation*}
|D\phi_0(x)|\geq \bm{e}\cdot D \phi_0(x)\geq\frac{\phi_0(x)-\phi_0(x_0)}{|x-x_0|}\geq \frac{-\frac{\eta_0}{2}-(-\eta_0)}{l_0}=\frac{\eta_0}{2l_0},\ \ \ \ \forall x\in\Omega_1,
\end{equation*}
which proves the claim.

For any $x\in\Omega_1$, using the assumptions  $(\mathbf{A_2})$-$(\mathrm{i})$ and $(\mathbf{A_3})$ together with \eqref{dx-case2}, \eqref{lambda-min}, \eqref{lambda-max} and \eqref{claim-case21}, we infer that
\begin{align*}
	&\ \ \ \ F\left(\lambda_1(D^2 W_k),\cdots,\lambda_n(D^2 W_k)\right)\cdot\left[g(x)(d(x))^{\alpha}\right]^{-1}\\
	&\geq \Lambda_0\left(\lambda_{\min}(D^2 W_k)\right)^{a}\left(\lambda_{\max}(D^2 W_k)\right)^{b}\cdot M_0^{-1}\left(d(x)\right)^{-\alpha}\\
	&\geq \Lambda_0  \left(\frac{C_1}{2}\underline{\rho_0}\right)^a\left(\frac{C_1\theta}{2\eta_0{}^\theta}\left(-\phi_0\right)^{\theta-1}\left(\frac{\eta_0}{2l_0}\right)^2\right)^bM_0^{-1}\left\|\phi_0\right\|^{\alpha}_{C^2(\overline{\Omega})}\cdot\left(-\phi_0(x)\right)^{-\alpha}\\
	&= \frac{1}{2^{a+3b}}C_1^{a+b} \Lambda_0 M_0^{-1} \underline{\rho_0}^a \theta^b\eta_0{}^{(2-\theta)b}l_0{}^{-2b}
\left\|\phi_0\right\|^{\alpha}_{C^2(\overline{\Omega})}\cdot\left(-\phi_0\right)^{b(\theta-1)-\alpha},\ \ \forall x\in\Omega_1.
\end{align*}

Let us take $\theta_0$ such that $b(\theta_0-1)-\alpha=0$, i.e. 
\begin{equation*}
\theta_0=\frac{b+\alpha}{b}\in(0,1).
\end{equation*}
 When $\theta=\theta_0$, it is now clear that
\begin{align*}
&\ \ \ \  F\left(\lambda_1(D^2 W_k),\cdots,\lambda_n(D^2 W_k)\right)\cdot\left[g(x)(d(x))^{\alpha}\right]^{-1}\\
&	\geq \frac{1}{2^{a+3b}}C_1^{a+b} \Lambda_0 M_0^{-1} \underline{\rho_0}^a \theta_0{}^b\eta_0{}^{(2-\theta_0)b}l_0{}^{-2b}
	\left\|\phi_0\right\|^{\alpha}_{C^2(\overline{\Omega})},\ \ \forall x\in\Omega_1.
\end{align*}
Hence there exists some constant $C_{2,1}>0$ such that for any $C_1\geq C_{2,1}$, we have 
\[F\left(\lambda_1(D^2 W_k),\cdots,\lambda_n(D^2 W_k)\right)\cdot\left[g(x)(d(x))^{\alpha}\right]^{-1}>1,\ \ \forall x\in\Omega_1,\]
i.e. 
\begin{equation}\label{case2-F}
	F\left(\lambda_1(D^2 W_k),\cdots,\lambda_n(D^2 W_k)\right) >g(x)(d(x))^{\alpha},\ \ \forall x\in\Omega_1.
\end{equation}

\subsection*{Case 2-2: In $\Omega_2$, $D\phi_0$  can be degenerate.} 
By virtue of \eqref{phi-0-1} and \eqref{phi-0-2}, 
$\phi_0$ indeed takes its minimum at $x_0\in\Omega_2$. Then we have $D\phi_0(x_0)=0$, and hence we need to reconsider $\lambda_{\max}(D^2W_k)$ in $\Omega_2$:
\begin{equation}\label{lambda-max-2}
\lambda_{\max}(D^2W_k)=\max_{|\bm{\xi}|=1}\,\bm{\xi}^TD^2W_k\,\bm{\xi} \geq \frac{C_1}{2}\max_{|\bm{\xi}|=1}\,\bm{\xi}^TD^2\phi_0\,\bm{\xi}\geq\frac{C_1}{2}\underline{\rho_0}.
\end{equation}

For any $x\in\Omega_2$, we note that 
\begin{equation*} 
	 -\phi_0(x)>\frac{\eta_0}{2}.
\end{equation*}
Then using the assumptions  $(\mathbf{A_2})$-$(\mathrm{i})$ and $(\mathbf{A_3})$ together with \eqref{dx-case2}, \eqref{lambda-min} and \eqref{lambda-max-2}, we obtain 
\begin{align*}
	&\ \ \ \ F\left(\lambda_1(D^2 W_k),\cdots,\lambda_n(D^2 W_k)\right)\cdot\left[g(x)(d(x))^{\alpha}\right]^{-1}\\
	&\geq \Lambda_0\left(\lambda_{\min}(D^2 W_k)\right)^{a}\left(\lambda_{\max}(D^2 W_k)\right)^{b}\cdot M_0^{-1}\left(d(x)\right)^{-\alpha}\\
	&\geq \Lambda_0  \left(\frac{C_1}{2}\underline{\rho_0}\right)^a\left(\frac{C_1}{2}\underline{\rho_0}
	\right)^bM_0^{-1}\left\|\phi_0\right\|^{\alpha}_{C^2(\overline{\Omega})}\cdot\left(-\phi_0(x)\right)^{-\alpha}\\
	&\geq \Lambda_0  \left(\frac{C_1}{2}\underline{\rho_0}\right)^a\left(\frac{C_1}{2}\underline{\rho_0}
	\right)^bM_0^{-1}\left\|\phi_0\right\|^{\alpha}_{C^2(\overline{\Omega})}\cdot\left(\frac{\eta_0}{2}\right)^{-\alpha}\\
	&= \frac{1}{2^{a+b-\alpha}}C_1^{a+b} \Lambda_0 M_0^{-1} \underline{\rho_0}^{a+b} 
	\left\|\phi_0\right\|^{\alpha}_{C^2(\overline{\Omega})}\cdot\eta_0{}^{-\alpha},\ \ \forall x\in\Omega_2.
\end{align*}
Hence there exists some constant $C_{2,2}>0$ such that for any $C_1\geq C_{2,2}$, we have 
\[F\left(\lambda_1(D^2 W_k),\cdots,\lambda_n(D^2 W_k)\right)\cdot\left[g(x)(d(x))^{\alpha}\right]^{-1}>1,\ \ \forall x\in\Omega_2,\]
i.e. 
\begin{equation}\label{case2-F-2}
	F\left(\lambda_1(D^2 W_k),\cdots,\lambda_n(D^2 W_k)\right) >g(x)(d(x))^{\alpha},\ \ \forall x\in\Omega_2.
\end{equation}

Up to now, combining \eqref{case2-F} and \eqref{case2-F-2} together, we can conclude that for any $C_1\geq C_0:=\max\{C_{2,1},C_{2,2}\}$  (without loss of generality, we can take $C_1=C_0$ here), there holds
\begin{equation}\label{case2-F-3}
	F\left(\lambda_1(D^2 W_k),\cdots,\lambda_n(D^2 W_k)\right) >g(x)(d(x))^{\alpha},\ \ \forall x\in\Omega_1\cup\Omega_2=\Omega.
\end{equation}
Therefore we obtain that 
\begin{equation}\label{eq:W-case2-new}
	W_k=C_0\phi_0+\frac{C_0}{2(1+\theta_0)\eta_0{}^{\theta_0}}\left(-\phi_0\right)^{1+\theta_0}+\psi_k
\end{equation} is a classically strict sub-solution to \eqref{equation} over $\Omega$.

By \eqref{eq:W-case2-new}, Lemma \ref{lemma:approximation}  and Lemma \ref{lemma:uniform-convex}, we have $W_k\in C(\overline{\Omega})$ and
\begin{equation*}
	W_k\big|_{\partial\Omega}=\psi_k\big|_{\partial\Omega}=\varphi.
\end{equation*}
Now, by virtue of \eqref{case2-F-3} and Lemma \ref{lemma2}, the equation \eqref{equation}    admits a convex viscosity solution $u\in C(\overline{\Omega})$ satisfying $u\big|_{\partial\Omega}= \varphi$, i.e. the problem \eqref{equation}-\eqref{bcondition} admits a viscosity solution $u\in C(\overline{\Omega})$.  Using \eqref{case2-F-3}  again, we can obtain from the Comparison Principle (Lemma \ref{lemma1}) that  
\begin{equation*} 
	W_k(x)\leq  u(x), \ \ \ \  \forall x\in\overline{\Omega}.
\end{equation*}
Then by the definition \eqref{eq:W-case2-new} and Lemma \ref{lemma:dx}, we can derive
\begin{align*}
	u(x)
	&\geq C_0\phi_0(x)+\frac{C_0}{2(1+\theta_0)\eta_0{}^{\theta_0}}\left(-\phi_0(x)\right)^{1+\theta_0}+\psi_k(x)\\
	&\geq C_0\phi_0(x)+\psi_k(x)\\
	&\geq C_0\left\|\phi_0\right\|_{C^2(\overline{\Omega})}\cdot\left(-d(x)\right)+\psi_k(x), \ \ \ \  \forall x\in\overline{\Omega}.
\end{align*}
Denote 
$\overline{C_0}=C_0\left\|\phi_0\right\|_{C^2(\overline{\Omega})} $. Thus, by letting $k\to\infty$, we arrive at the result that
\begin{equation*}
	u(x)\geq \psi(x)-\overline{C_0}\,d(x), \ \ \forall x\in\overline{\Omega}.
\end{equation*}

To conclude, the above Case 1 and Case 2 lead us to Theorem \ref{thm1} (i).

\subsection{Case 3: $\alpha=-b$}

We are going to construct a  classically strict  sub-solution in the following form:
\begin{equation}\label{eq:W-case3}
	W_k=-C(-\overline{\phi_0}) \left[-\log(-\overline{\phi_0})\right]^\theta
	+\psi_k,
\end{equation}
where $C>0$  and $\theta\in (0,1]$ are constants to be determined, and 
\begin{equation*}
\overline{\phi_0}=\frac{\phi_0}{3\eta_0}.
\end{equation*}
By virtue of \eqref{phi-0-1} and \eqref{phi-0-2}, we note that $-\overline{\phi_0}(x)\in (0,\frac{1}{3})$ for any $x\in\Omega$. 
Along with Lemma \ref{lemma:dx}, we also have
\begin{align*}
\left(d(x)\right)^{b}
&\geq\left\|\phi_0\right\|^{-b}_{C^2(\overline{\Omega})}\cdot\left(-\phi_0(x)\right)^{b},\\
&\geq\left\|\phi_0\right\|^{-b}_{C^2(\overline{\Omega})}\cdot\left(-3\eta_0\overline{\phi_0}(x)\right)^{b},\ \ \ \ \forall x\in\Omega.\stepcounter{equation}\tag{\theequation}\label{dx-case3}
\end{align*}

By \eqref{eq:W-case3}, direct computations give rise to 
\begin{align*}
DW_k&=\left\{C\left[-\log(-\overline{\phi_0})\right]^\theta-C\theta\left[-\log(-\overline{\phi_0})\right]^{\theta-1}\right\}\cdot\nabla\overline{\phi_0}+D\psi_k,\\
D^2W_k&=\left\{C\left[-\log(-\overline{\phi_0})\right]^\theta-C\theta\left[-\log(-\overline{\phi_0})\right]^{\theta-1}\right\}\cdot D^2\overline{\phi_0}\\
&\ \ \ \ +\left\{C\theta(-\overline{\phi_0})^{-1}\left[-\log(-\overline{\phi_0})\right]^{\theta-1}-C\theta(\theta-1)(-\overline{\phi_0})^{-1}\left[-\log(-\overline{\phi_0})\right]^{\theta-2}\right\}\cdot D\overline{\phi_0} \otimes D\overline{\phi_0} +\nabla^2\psi_k.\stepcounter{equation}\tag{\theequation}\label{D2-case31}
\end{align*}

In order to  bound $-\log(-\overline{\phi_0})$, we also split $\Omega$ as two disjoint domains.
\begin{equation*}
	\Omega:=\Omega_3\cup\Omega_4.
\end{equation*}
Here, we use
\begin{equation*}
	\Omega_3:=\left\{x\in\Omega:-\overline{\eta_0}\leq\phi_0(x)<0\right\}
\end{equation*}
to represent the domain near the boundary,
and use
\begin{equation*}
	\Omega_4:=\left\{x\in\Omega:-\eta_0\leq\phi_0(x)<-\overline{\eta_0}\right\}
\end{equation*}
to represent the domain away from the boundary, where 
\begin{equation*}
\overline{\eta_0}=\min\{2e^{-2\theta}\eta_0,\frac{\eta_0}{2}\}.
\end{equation*}

\subsection*{Case 3-1: In $\Omega_3$,  $D\phi_0$ is nondegenerate.}

 We also claim that 
\begin{equation}\label{claim-case31}
	|D\phi_0(x)|\geq\frac{\eta_0}{2l_0},\ \ \ \ \forall x\in\Omega_3.
\end{equation}
In fact, this claim can be proved in a similar manner as \eqref{claim-case21} with only   one modification due to the definition of $\Omega_3$:
\begin{equation*}
	|D\phi_0(x)| \geq\frac{\phi_0(x)-\phi_0(x_0)}{|x-x_0|}\geq \frac{-\overline{\eta_0}-(-\eta_0)}{l_0}\geq \frac{-\frac{\eta_0}{2}-(-\eta_0)}{l_0}=\frac{\eta_0}{2l_0},\ \ \ \ \forall x\in\Omega_3.
\end{equation*}

For $x\in\Omega_3$, it is obvious that $-2e^{-2\theta}\eta_0\leq\phi_0<0$, then $0<-\overline{\phi_0}\leq e^{-2\theta}$, and hence
\begin{equation}\label{condition-1}
-\log(-\overline{\phi_0})\geq 2\theta.
\end{equation}
It follows that 
\begin{equation*}
C\left[-\log(-\overline{\phi_0})\right]^\theta=C\left[-\log(-\overline{\phi_0})\right]\cdot \left[-\log(-\overline{\phi_0})\right]^{\theta-1}\geq 2C\theta\left[-\log(-\overline{\phi_0})\right]^{\theta-1},
\end{equation*}
which implies
\begin{equation*}
C\theta\left[-\log(-\overline{\phi_0})\right]^{\theta-1}\leq \frac{1}{2} C\left[-\log(-\overline{\phi_0})\right]^\theta.
\end{equation*}
Together with \eqref{condition-1}, this yields
\begin{equation}\label{condition-2}
\frac{1}{2} C\left[-\log(-\overline{\phi_0})\right]^\theta\leq C\left[-\log(-\overline{\phi_0})\right]^\theta-C\theta\left[-\log(-\overline{\phi_0})\right]^{\theta-1} \leq C\left[-\log(-\overline{\phi_0})\right]^\theta.
\end{equation}
Since $\theta\in (0,1]$, we also have 
\begin{equation}\label{condition-3}
	-C\theta(\theta-1)(-\overline{\phi_0})^{-1}\left[-\log(-\overline{\phi_0})\right]^{\theta-2}\geq 0.
\end{equation}

Based on Remark \ref{remark:rho} and Remark \ref{remark:product}, we note that 
\begin{align*}
 D^2\overline{\phi_0}&=\frac{1}{3\eta_0}D^2\phi_0>0,\\
 D\overline{\phi_0} \otimes D\overline{\phi_0}&= \frac{1}{9\eta_0{}^2}  D \phi_0\otimes D \phi_0\geq 0.
\end{align*}
Hence,  
using \eqref{condition-2} and \eqref{condition-3}  leads us to 
\begin{align*}
	D^2W_k\geq \frac{1}{2} C\left[-\log(-\overline{\phi_0})\right]^\theta\cdot D^2\overline{\phi_0}  + C\theta(-\overline{\phi_0})^{-1}\left[-\log(-\overline{\phi_0})\right]^{\theta-1} \cdot D\overline{\phi_0} \otimes D\overline{\phi_0} +\nabla^2\psi_k.
\end{align*}

Now we are ready to deduce that
\begin{align*}
	\lambda_{\min}(D^2W_k)&=\min_{|\bm{\xi}|=1}\,\bm{\xi}^TD^2W_k\,\bm{\xi}\\
&\geq \min_{|\bm{\xi}|=1}\,\bm{\xi}^T\left(\frac{1}{2} C\left[-\log(-\overline{\phi_0})\right]^\theta\cdot D^2\overline{\phi_0}\right)\bm{\xi}\\
& = \frac{1}{2} C\left[-\log(-\overline{\phi_0})\right]^\theta\min_{|\bm{\xi}|=1}\,\bm{\xi}^TD^2\overline{\phi_0}\, \bm{\xi}\\
&\geq \frac{1}{6\eta_0} C\left[-\log(-\overline{\phi_0})\right]^\theta\min_{|\bm{\xi}|=1}\,\bm{\xi}^TD^2 \phi_0\, \bm{\xi}\\
&\geq\frac{1}{6\eta_0} C\left[-\log(-\overline{\phi_0})\right]^\theta\underline{\rho_0}.\stepcounter{equation}\tag{\theequation}\label{lambda-min-case31}
\end{align*}
By virtue of 
Remark \ref{remark:product} and 
\eqref{claim-case31}, we also obtain
\begin{align*}
	\lambda_{\max}(D^2W_k)&=\max_{|\bm{\xi}|=1}\,\bm{\xi}^TD^2W_k\,\bm{\xi}\\
	&\geq \max_{|\bm{\xi}|=1}\,\bm{\xi}^T\left(C\theta(-\overline{\phi_0})^{-1}\left[-\log(-\overline{\phi_0})\right]^{\theta-1} \cdot D\overline{\phi_0} \otimes D\overline{\phi_0}\right)\bm{\xi}\\
	& = C\theta(-\overline{\phi_0})^{-1}\left[-\log(-\overline{\phi_0})\right]^{\theta-1}\max_{|\bm{\xi}|=1}\,\bm{\xi}^T D\overline{\phi_0} \otimes D\overline{\phi_0}\, \bm{\xi}\\
	&\geq\frac{1}{9\eta_0{}^2} C\theta(-\overline{\phi_0})^{-1}\left[-\log(-\overline{\phi_0})\right]^{\theta-1}   \max_{|\bm{\xi}|=1}\,\bm{\xi}^T D\phi_0\otimes D\phi_0\,\bm{\xi}\\ 
	&=\frac{1}{9\eta_0{}^2} C\theta(-\overline{\phi_0})^{-1}\left[-\log(-\overline{\phi_0})\right]^{\theta-1}  |D \phi_0|^2\\
	&\geq \frac{1}{36l_0{}^2} C\theta(-\overline{\phi_0})^{-1}\left[-\log(-\overline{\phi_0})\right]^{\theta-1}.\stepcounter{equation}\tag{\theequation}\label{lambda-max-case31}
\end{align*}

Using the assumptions  $(\mathbf{A_2})$-$(\mathrm{i})$ and $(\mathbf{A_3})$ together with \eqref{dx-case3}, \eqref{claim-case31}, \eqref{lambda-min-case31} and  \eqref{lambda-max-case31}, we derive
\begin{align*}
	&\ \ \ \ F\left(\lambda_1(D^2 W_k),\cdots,\lambda_n(D^2 W_k)\right)\cdot\left[g(x)(d(x))^{\alpha}\right]^{-1}\\
	&\geq \Lambda_0\left(\lambda_{\min}(D^2 W_k)\right)^{a}\left(\lambda_{\max}(D^2 W_k)\right)^{b}\cdot M_0^{-1}\left(d(x)\right)^{b}\\
	&\geq \Lambda_0  \left\{\frac{1}{6\eta_0} C\left[-\log(-\overline{\phi_0})\right]^\theta\underline{\rho_0}\right\}^a\left\{\frac{1}{36l_0{}^2} C\theta(-\overline{\phi_0})^{-1}\left[-\log(-\overline{\phi_0})\right]^{\theta-1}
\right\}^bM_0^{-1}\left\|\phi_0\right\|^{-b}_{C^2(\overline{\Omega})}\cdot\left(-3\eta_0\overline{\phi_0}(x)\right)^{b}\\
	&= \frac{1}{6^a\times12^b}C^{a+b} \Lambda_0 M_0^{-1} \underline{\rho_0}^a \theta^b 
	\eta_0{}^{b-a}l_0{}^{-2b}\left\|\phi_0\right\|^{-b}_{C^2(\overline{\Omega})}\left[-\log(-\overline{\phi_0})\right]^{(a+b)\theta-b},\ \ \forall x\in\Omega_3.
\end{align*}
Let us take $\theta_0$ such that $(a+b)\theta_0-b=0$, i.e. 
\begin{equation*}
	\theta_0=\frac{b}{a+b}\in (0,1].
\end{equation*}
When $\theta=\theta_0$, it is now clear that
\begin{align*}
	&\ \ \ \  F\left(\lambda_1(D^2 W_k),\cdots,\lambda_n(D^2 W_k)\right)\cdot\left[g(x)(d(x))^{\alpha}\right]^{-1}\\
	&	\geq \frac{1}{6^a\times12^b}C^{a+b} \Lambda_0 M_0^{-1} \underline{\rho_0}^a \theta_0{}^b 
	\eta_0{}^{b-a}l_0{}^{-2b}\left\|\phi_0\right\|^{-b}_{C^2(\overline{\Omega})},\ \ \forall x\in\Omega_3.
\end{align*} 
Hence there exists some constant $C_{3,1}>0$ such that for any $C\geq C_{3,1}$, we have 
\[F\left(\lambda_1(D^2 W_k),\cdots,\lambda_n(D^2 W_k)\right)\cdot\left[g(x)(d(x))^{\alpha}\right]^{-1}>1,\ \ \forall x\in\Omega_3,\]
i.e. 
\begin{equation}\label{case3-F-1}
	F\left(\lambda_1(D^2 W_k),\cdots,\lambda_n(D^2 W_k)\right) >g(x)(d(x))^{\alpha},\ \ \forall x\in\Omega_3.
\end{equation}

\subsection*{Case 3-2: In $\Omega_4$, $D\phi_0$  can be degenerate.} 
By \eqref{phi-0-1} and \eqref{phi-0-2}, 
$\phi_0$ indeed takes its minimum at $x_0\in\Omega_4$. It follows that $D\phi_0(x_0)=0$ and therefore we need to reconsider $\lambda_{\max}(D^2W_k)$.

Now we still take $\theta=\theta_0=\frac{b}{a+b}\in (0,1]$. 
For $x\in\Omega_4$, it is clear that $-\eta_0\leq\phi_0<-\overline{\eta_0}=-\min\{2e^{-2\theta_0}\eta_0,\frac{\eta_0}{2}\}$, then
\begin{equation}\label{condition-4-2}
	\min\left\{e^{-2\theta_0},\frac{1}{6}\right\}<-\overline{\phi_0}\leq \frac{1}{3},
\end{equation}  and hence
\begin{equation}\label{condition-4}
	-\log(-\overline{\phi_0})\geq -\log\frac{1}{3}=\log 3.
\end{equation}
We can infer that
\begin{align*}
	C\left[-\log(-\overline{\phi_0})\right]^{\theta_0}&=C\left[-\log(-\overline{\phi_0})\right]\cdot \left[-\log(-\overline{\phi_0})\right]^{\theta_0-1}\\
	&\geq C\log 3\cdot\left[-\log(-\overline{\phi_0})\right]^{\theta_0-1}\\
	&=\frac{\log 3}{\theta_0}\cdot C\theta_0\left[-\log(-\overline{\phi_0})\right]^{\theta_0-1},
\end{align*}
which yields
\begin{equation*}
	C\theta_0\left[-\log(-\overline{\phi_0})\right]^{\theta_0-1}\leq \frac{\theta_0}{\log 3}\cdot C\left[-\log(-\overline{\phi_0})\right]^{\theta_0}.
\end{equation*}
Using  \eqref{condition-4} again, we obtain 
\begin{equation}\label{condition-5}
 C\left[-\log(-\overline{\phi_0})\right]^{\theta_0}-C\theta\left[-\log(-\overline{\phi_0})\right]^{\theta_0-1} \geq\left(1-\frac{\theta_0}{\log 3}\right)\cdot C\left[-\log(-\overline{\phi_0})\right]^{\theta_0}.
\end{equation}
Here, $1-\frac{\theta_0}{\log 3}>0$ due to $\theta_0\in(0,1]$ and $\log 3>1$. Hence,  
using \eqref{D2-case31}, \eqref{condition-4} and \eqref{condition-5}  leads us to 
\begin{align*}
	D^2W_k\geq \left(1-\frac{\theta_0}{\log 3}\right)\cdot C\left[-\log(-\overline{\phi_0})\right]^{\theta_0}\cdot D^2\overline{\phi_0}  + C\theta_0(-\overline{\phi_0})^{-1}\left[-\log(-\overline{\phi_0})\right]^{\theta_0-1} \cdot D\overline{\phi_0} \otimes D\overline{\phi_0} +\nabla^2\psi_k.
\end{align*}

Now we are ready to estimate $\lambda_{\min}(D^2W_k)$ and $\lambda_{\max}(D^2W_k)$ in $\Omega_4$ in the same manner:
\begin{align*} 
		\lambda_{\max}(D^2W_k)&\geq	\lambda_{\min}(D^2W_k)=\min_{|\bm{\xi}|=1}\,\bm{\xi}^TD^2W_k\,\bm{\xi}\\
	&\geq \min_{|\bm{\xi}|=1}\,\bm{\xi}^T\left(\left(1-\frac{\theta_0}{\log 3}\right)\cdot C\left[-\log(-\overline{\phi_0})\right]^{\theta_0}\cdot D^2\overline{\phi_0}\right)\bm{\xi}\\
	& = \left(1-\frac{\theta_0}{\log 3}\right)\cdot C\left[-\log(-\overline{\phi_0})\right]^{\theta_0}\min_{|\bm{\xi}|=1}\,\bm{\xi}^TD^2\overline{\phi_0}\, \bm{\xi}\\
	&= \frac{1}{3\eta_0}\left(1-\frac{\theta_0}{\log 3}\right)\cdot C\left[-\log(-\overline{\phi_0})\right]^{\theta_0}\min_{|\bm{\xi}|=1}\,\bm{\xi}^TD^2 \phi_0\, \bm{\xi}\\
	&\geq\frac{1}{3\eta_0}\left(1-\frac{\theta_0}{\log 3}\right)\cdot C\left[-\log(-\overline{\phi_0})\right]^{\theta_0}\underline{\rho_0}.\stepcounter{equation}\tag{\theequation}\label{lambda-min-case32}
\end{align*}

Using the assumptions  $(\mathbf{A_2})$-$(\mathrm{i})$ and $(\mathbf{A_3})$ together with \eqref{dx-case3},   \eqref{lambda-min-case32},  \eqref{condition-4-2} and \eqref{condition-4}, we derive 
\begin{align*}
	&\ \ \ \ F\left(\lambda_1(D^2 W_k),\cdots,\lambda_n(D^2 W_k)\right)\cdot\left[g(x)(d(x))^{\alpha}\right]^{-1}\\
	&\geq \Lambda_0\left(\lambda_{\min}(D^2 W_k)\right)^{a}\left(\lambda_{\max}(D^2 W_k)\right)^{b}\cdot M_0^{-1}\left(d(x)\right)^{b}\\
	&\geq \Lambda_0  \left\{\frac{1}{3\eta_0}\left(1-\frac{\theta_0}{\log 3}\right)\cdot C\left[-\log(-\overline{\phi_0})\right]^{\theta_0}\underline{\rho_0}\right\}^a\left\{\frac{1}{3\eta_0}\left(1-\frac{\theta_0}{\log 3}\right)\cdot C\left[-\log(-\overline{\phi_0})\right]^{\theta_0}\underline{\rho_0}
	\right\}^b\\
	&\ \ \ \ \cdot M_0^{-1}\left\|\phi_0\right\|^{-b}_{C^2(\overline{\Omega})}\cdot\left(-3\eta_0\overline{\phi_0}(x)\right)^{b}\\
	&= \frac{1}{3^{a}}\left(1-\frac{\theta_0}{\log 3}\right)^{a+b}\cdot C^{a+b} \Lambda_0 M_0^{-1} \underline{\rho_0}^{a+b}
	\eta_0{}^{-a} \left\|\phi_0\right\|^{-b}_{C^2(\overline{\Omega})}\left[-\log(-\overline{\phi_0})\right]^{(a+b)\theta_0}\cdot\left(-\overline{\phi_0}\right)^{b}\\
	&\geq \frac{1}{3^{a}}\left(1-\frac{\theta_0}{\log 3}\right)^{a+b}\cdot C^{a+b} \Lambda_0 M_0^{-1} \underline{\rho_0}^{a+b}
	\eta_0{}^{-a} \left\|\phi_0\right\|^{-b}_{C^2(\overline{\Omega})}\left(\log 3\right)^{(a+b)\theta_0}\cdot\left(\min\left\{e^{-2\theta_0},\frac{1}{6}\right\}\right)^{b},\ \ \forall x\in\Omega_4.
\end{align*}
Hence there exists some constant $C_{3,2}>0$ such that for any $C\geq C_{3,2}$, we have 
\[F\left(\lambda_1(D^2 W_k),\cdots,\lambda_n(D^2 W_k)\right)\cdot\left[g(x)(d(x))^{\alpha}\right]^{-1}>1,\ \ \forall x\in\Omega_4,\]
i.e. 
\begin{equation}\label{case3-F-2}
	F\left(\lambda_1(D^2 W_k),\cdots,\lambda_n(D^2 W_k)\right) >g(x)(d(x))^{\alpha},\ \ \forall x\in\Omega_4.
\end{equation}

Putting \eqref{case3-F-1} and \eqref{case3-F-2} together, we can summarize that for any $C\geq C_0:=\max\{C_{3,1},C_{3,2}\}$  (without loss of generality, we can take $C=C_0$ here), there holds
\begin{equation}\label{case3-F-3}
	F\left(\lambda_1(D^2 W_k),\cdots,\lambda_n(D^2 W_k)\right) >g(x)(d(x))^{\alpha},\ \ \forall x\in\Omega_3\cup\Omega_4=\Omega.
\end{equation}
Now we obtain that 
\begin{equation}\label{eq:W-case3-new}
	W_k=-C_0(-\overline{\phi_0}) \left[-\log(-\overline{\phi_0})\right]^{\theta_0}
	+\psi_k
\end{equation}
 is a classically strict sub-solution to \eqref{equation} over $\Omega$.

For any $x\in\Omega$, by Lemma \ref{lemma:uniform-convex}, we note that 
\begin{equation*}
\lim_{x\to\partial\Omega}	\left(-\overline{\phi_0}\right)=\lim_{x\to\partial\Omega}\left(-\frac{\phi_0}{3\eta_0}\right)= 0,
\end{equation*}
which further yields
\begin{align*}
\lim_{x\to\partial\Omega}(-\overline{\phi_0}) \left[-\log(-\overline{\phi_0})\right]^{\theta_0}=\lim_{t_1\to 0}t_1 \left(-\log t_1\right)^{\theta_0}=\lim_{t_1\to 0}\frac{\left(\log \frac{1}{t_1}\right)^{\theta_0}}{\frac{1}{t_1}}=\lim_{t_2\to +\infty}\frac{\left(\log t_2\right)^{\theta_0}}{t_2}
=0.
\end{align*}
Then by the definition \eqref{eq:W-case3-new}, Lemma \ref{lemma:approximation} and Lemma \ref{lemma:uniform-convex}, we have  $W_k\in C(\overline{\Omega})$ and
\begin{equation*}
	W_k\big|_{\partial\Omega}=\psi_k\big|_{\partial\Omega}=\varphi.
\end{equation*}

Thus, using \eqref{case3-F-3} and Lemma \ref{lemma2}, the equation \eqref{equation}    admits a convex viscosity solution $u\in C(\overline{\Omega})$ satisfying $u\big|_{\partial\Omega}= \varphi$. Using \eqref{case3-F-3}  again, we can obtain from the Comparison Principle (Lemma \ref{lemma1}) that  
\begin{equation*} 
	W_k(x)\leq  u(x), \ \ \ \  \forall x\in\overline{\Omega}.
\end{equation*}
Then by  definition \eqref{eq:W-case3-new}, we can derive
\begin{equation*}
	u(x)
\geq -C_0(-\overline{\phi_0}(x)) \left[-\log(-\overline{\phi_0}(x))\right]^{\theta_0}+\psi_k(x), \ \ \ \  \forall x\in\overline{\Omega}.
\end{equation*}

By virtue of Lemma \ref{lemma:dx}, we note that 
\begin{equation*}
	-\overline{\phi_0}(x)=\frac{1}{3\eta_0}(-\phi_0(x))\leq \frac{1}{3\eta_0}	\left\|\phi_0\right\|_{C^2(\overline{\Omega})}\cdot d(x),\ \ \ \  \forall x\in\overline{\Omega},
\end{equation*}
and 
\begin{equation*}
	-\overline{\phi_0}(x)=\frac{1}{3\eta_0}(-\phi_0(x))\geq \frac{1}{3\eta_0}\frac{\eta_0}{\operatorname{diam}(\Omega)}\cdot d(x)=\frac{1}{3\operatorname{diam}(\Omega)}\cdot d(x),\ \ \ \  \forall x\in\overline{\Omega}.
\end{equation*}
The latter also gives 
\begin{equation*}
\left[-\log(-\overline{\phi_0}(x))\right]^{\theta_0}\leq \left[-\log\left(\frac{1}{3\operatorname{diam}(\Omega)}\cdot d(x)\right)\right]^{\theta_0},\ \ \ \  \forall x\in\overline{\Omega}.
\end{equation*}
Thus it follows that 
\begin{equation*}
	u(x)
	\geq -\frac{C_0}{3\eta_0}	\left\|\phi_0\right\|_{C^2(\overline{\Omega})}\cdot d(x)\cdot \left[-\log\left(\frac{1}{3\operatorname{diam}(\Omega)}\cdot d(x)\right)\right]^{\theta_0}+\psi_k(x), \ \ \ \  \forall x\in\overline{\Omega}.
\end{equation*}
Denote 
$\overline{C_0}=\frac{C_0}{3\eta_0}	\left\|\phi_0\right\|_{C^2(\overline{\Omega})}$ and $\overline{C_1}=\frac{1}{3\operatorname{diam}(\Omega)}$. Therefore, by letting $k\to\infty$, we can conclude that
\begin{equation*}
	u(x)
	\geq \psi(x)-\overline{C_0} d(x) \left[-\log\left(\overline{C_1} d(x)\right)\right]^{\theta_0}, \ \ \ \  \forall x\in\overline{\Omega},
\end{equation*}
where $\theta_0=\frac{b}{a+b}$. 
This gives rise to Theorem \ref{thm1} (ii).

\subsection{Case 4: $\alpha\in(-a-2b,-b)$}

Now we consider
\begin{equation*} 
	W_k=-C\left(-\phi_0\right)^{\theta}+\psi_k,
\end{equation*}
where $C\geq 0$  and $\theta\in (0,1)$ are constants to be determined. 

It follows that 
\begin{align*} 
	DW_k&=C\theta\left(-\phi_0\right)^{\theta-1}D\phi_0+D\psi_k,\\
	D^2	W_k&=C\theta\left(-\phi_0\right)^{\theta-1}D^2\phi_0+C\theta(1-\theta)\left(-\phi_0\right)^{\theta-2}D\phi_0\otimes D\phi_0+D^2\psi_k.
\end{align*}
Analogously, using Remark \ref{remark:phi-k},  Lemma \ref{lemma:eigenvalue},  Remark \ref{remark:rho} and Remark \ref{remark:product}, we can derive that 
\begin{align*}
	\lambda_{\min}(D^2W_k)&=\min_{|\bm{\xi}|=1}\,\bm{\xi}^TD^2W_k\,\bm{\xi} \geq C\theta\left(-\phi_0\right)^{\theta-1}\min_{|\bm{\xi}|=1}\,\bm{\xi}^TD^2\phi_0\,\bm{\xi}\geq C\theta\left(-\phi_0\right)^{\theta-1}\underline{\rho_0},\stepcounter{equation}\tag{\theequation}\label{lambda-min-case41}\\
	\lambda_{\max}(D^2W_k)&=\max_{|\bm{\xi}|=1}\,\bm{\xi}^TD^2W_k\,\bm{\xi} \geq C\theta(1-\theta)\left(-\phi_0\right)^{\theta-2}\max_{|\bm{\xi}|=1}\,\bm{\xi}^TD\phi_0\otimes D\phi_0\,\bm{\xi}= C\theta(1-\theta)\left(-\phi_0\right)^{\theta-2}|D\phi_0|^2.\stepcounter{equation}\tag{\theequation}\label{lambda-max-case41}
\end{align*}

Similar to \eqref{dx-case2}, we also have  \begin{equation}\label{dx-case4}
	\left(d(x)\right)^{-\alpha}\geq\left\|\phi_0\right\|^{\alpha}_{C^2(\overline{\Omega})}\cdot\left(-\phi_0(x)\right)^{-\alpha},\ \ \ \ \forall x\in\Omega.
\end{equation}
In the same manner as Case 2, we also split $\Omega$ as 
\begin{equation*}
	\Omega=\Omega_1\cup\Omega_2,
\end{equation*}
where
\begin{equation*}
	\Omega_1=\left\{x\in\Omega:-\frac{\eta_0}{2}\leq\phi_0(x)<0\right\},\ \ \ \ \Omega_2=\left\{x\in\Omega:-\eta_0\leq\phi_0(x)<-\frac{\eta_0}{2}\right\}.
\end{equation*} 

\subsection*{Case 4-1: In $\Omega_1$, $D\phi_0$ is nondegenerate.} By \eqref{claim-case21}, we also have
\begin{equation}\label{claim-case41}
	|D\phi_0(x)|\geq\frac{\eta_0}{2l_0},\ \ \ \ \forall x\in\Omega_1.
\end{equation}

For any $x\in\Omega_1$, using the assumptions  $(\mathbf{A_2})$-$(\mathrm{i})$ and $(\mathbf{A_3})$ along with  \eqref{lambda-min-case41}, \eqref{lambda-max-case41}, \eqref{dx-case4} and \eqref{claim-case41}, we obtain 
\begin{align*}
	&\ \ \ \ F\left(\lambda_1(D^2 W_k),\cdots,\lambda_n(D^2 W_k)\right)\cdot\left[g(x)(d(x))^{\alpha}\right]^{-1}\\
	&\geq \Lambda_0\left(\lambda_{\min}(D^2 W_k)\right)^{a}\left(\lambda_{\max}(D^2 W_k)\right)^{b}\cdot M_0^{-1}\left(d(x)\right)^{-\alpha}\\
	&\geq \Lambda_0  \left(C\theta\left(-\phi_0\right)^{\theta-1}\underline{\rho_0}\right)^a\left(C\theta(1-\theta)\left(-\phi_0\right)^{\theta-2}\left(\frac{\eta_0}{2l_0}\right)^2\right)^bM_0^{-1}\left\|\phi_0\right\|^{\alpha}_{C^2(\overline{\Omega})}\cdot\left(-\phi_0(x)\right)^{-\alpha}\\
	&= \frac{1}{2^{2b}}C^{a+b} \Lambda_0 M_0^{-1} \underline{\rho_0}^a \theta^{a+b}(1-\theta)^b\eta_0{}^{2b}l_0{}^{-2b}
	\left\|\phi_0\right\|^{\alpha}_{C^2(\overline{\Omega})}\cdot\left(-\phi_0\right)^{(a+b)\theta-(a+2b+\alpha)},\ \ \forall x\in\Omega_1.
\end{align*}

Let us take $\theta_0$ such that $(a+b)\theta_0-(a+2b+\alpha)=0$, i.e. 
\begin{equation*}
	\theta_0=\frac{a+2b+\alpha}{a+b}\in(0,1).
\end{equation*}
When $\theta=\theta_0$, it is now clear that
\begin{align*}
	&\ \ \ \  F\left(\lambda_1(D^2 W_k),\cdots,\lambda_n(D^2 W_k)\right)\cdot\left[g(x)(d(x))^{\alpha}\right]^{-1}\\
	&	\geq \frac{1}{2^{2b}}C^{a+b} \Lambda_0 M_0^{-1} \underline{\rho_0}^a \theta^{a+b}(1-\theta)^b\eta_0{}^{2b}l_0{}^{-2b}
	\left\|\phi_0\right\|^{\alpha}_{C^2(\overline{\Omega})},\ \ \forall x\in\Omega_1.
\end{align*}
Hence there exists some constant $C_{4,1}>0$ such that for any $C\geq C_{4,1}$, we have 
\[F\left(\lambda_1(D^2 W_k),\cdots,\lambda_n(D^2 W_k)\right)\cdot\left[g(x)(d(x))^{\alpha}\right]^{-1}>1,\ \ \forall x\in\Omega_1,\]
i.e. 
\begin{equation}\label{case4-F}
	F\left(\lambda_1(D^2 W_k),\cdots,\lambda_n(D^2 W_k)\right) >g(x)(d(x))^{\alpha},\ \ \forall x\in\Omega_1.
\end{equation}

\subsection*{Case 4-2: In $\Omega_2$, $D\phi_0$  can be degenerate.} Now we still take $\theta_0=\frac{a+2b+\alpha}{a+b}\in(0,1)$. 

Since $D\phi_0(x_0)=0$ at $x_0\in\Omega_2$, we also need to re-estimate $\lambda_{\max}(D^2W_k)$ in $\Omega_2$:
\begin{align*}
\lambda_{\max}(D^2W_k)&=\max_{|\bm{\xi}|=1}\,\bm{\xi}^TD^2W_k\,\bm{\xi}\\
& \geq \max_{|\bm{\xi}|=1}\,\bm{\xi}^T\left(C\theta_0\left(-\phi_0\right)^{\theta_0-1}D^2\phi_0\right)\bm{\xi}\\
&=C\theta_0\left(-\phi_0\right)^{\theta_0-1}\max_{|\bm{\xi}|=1}\,\bm{\xi}^TD^2\phi_0\,\bm{\xi}\\
&\geq C\theta_0\left(-\phi_0\right)^{\theta_0-1}\underline{\rho_0}.\stepcounter{equation}\tag{\theequation}\label{lambda-max-4}
\end{align*}
Then using the assumptions  $(\mathbf{A_2})$-$(\mathrm{i})$ and $(\mathbf{A_3})$ together with \eqref{dx-case4}, \eqref{lambda-min-case41} and \eqref{lambda-max-4}, we infer that 
\begin{align*}
	&\ \ \ \ F\left(\lambda_1(D^2 W_k),\cdots,\lambda_n(D^2 W_k)\right)\cdot\left[g(x)(d(x))^{\alpha}\right]^{-1}\\
	&\geq \Lambda_0\left(\lambda_{\min}(D^2 W_k)\right)^{a}\left(\lambda_{\max}(D^2 W_k)\right)^{b}\cdot M_0^{-1}\left(d(x)\right)^{-\alpha}\\
	&\geq \Lambda_0  \left(C\theta_0\left(-\phi_0\right)^{\theta_0-1}\underline{\rho_0}\right)^a\left(C\theta_0\left(-\phi_0\right)^{\theta_0-1}\underline{\rho_0}
	\right)^bM_0^{-1}\left\|\phi_0\right\|^{\alpha}_{C^2(\overline{\Omega})}\cdot\left(-\phi_0(x)\right)^{-\alpha}\\
	&=	 C^{a+b} \Lambda_0 M_0^{-1} \underline{\rho_0}^{a+b} \theta_0{}^{a+b}
	\left\|\phi_0\right\|^{\alpha}_{C^2(\overline{\Omega})}\cdot\left(-\phi_0\right)^{(a+b)\theta_0-(a+b+\alpha)},\ \ \forall x\in\Omega_2.	
\end{align*}
	
Since $\theta_0=\frac{a+2b+\alpha}{a+b}$, there holds
\begin{equation*}
	(a+b)\theta_0-(a+b+\alpha)=b.
\end{equation*}
For any $x\in\Omega_2$, we note that 
\begin{equation*} 
	-\phi_0(x)>\frac{\eta_0}{2}.
\end{equation*}
Thus we have 
\begin{equation*}
\left(-\phi_0\right)^{(a+b)\theta_0-(a+b+\alpha)}=\left(-\phi_0\right)^{b}>\left(\frac{\eta_0}{2}\right)^b.
\end{equation*}
Now  we can derive
\begin{align*}
  F\left(\lambda_1(D^2 W_k),\cdots,\lambda_n(D^2 W_k)\right)\cdot\left[g(x)(d(x))^{\alpha}\right]^{-1}\geq \frac{1}{2^b}	 C^{a+b} \Lambda_0 M_0^{-1} \underline{\rho_0}^a\overline{\rho_0}^b  \theta_0{}^{a+b}
	\left\|\phi_0\right\|^{\alpha}_{C^2(\overline{\Omega})}\cdot\eta_0{}^b,\ \ \forall x\in\Omega_2.	
\end{align*}
Hence there exists some constant $C_{4,2}>0$ such that for any $C\geq C_{4,2}$, we have 
\[F\left(\lambda_1(D^2 W_k),\cdots,\lambda_n(D^2 W_k)\right)\cdot\left[g(x)(d(x))^{\alpha}\right]^{-1}>1,\ \ \forall x\in\Omega_2,\]
i.e. 
\begin{equation}\label{case4-F-2}
	F\left(\lambda_1(D^2 W_k),\cdots,\lambda_n(D^2 W_k)\right) >g(x)(d(x))^{\alpha},\ \ \forall x\in\Omega_2.
\end{equation}

 Up to now, combining \eqref{case4-F} and \eqref{case4-F-2} together, we can conclude that for any $C\geq C_0:=\max\{C_{4,1},C_{4,2}\}$  (without loss of generality, we can take $C=C_0$ here), there holds
 \begin{equation}\label{case4-F-3}
 	F\left(\lambda_1(D^2 W_k),\cdots,\lambda_n(D^2 W_k)\right) >g(x)(d(x))^{\alpha},\ \ \forall x\in\Omega_1\cup\Omega_2=\Omega.
 \end{equation}
Then it is clear that  
\begin{equation}\label{eq:W-case4-new}
	W_k=-C_0\left(-\phi_0\right)^{\theta_0}+\psi_k
\end{equation} is a classically strict sub-solution to \eqref{equation} over $\Omega$.

Using the definition \eqref{eq:W-case4-new}, Lemma \ref{lemma:approximation} and Lemma \ref{lemma:uniform-convex}, we have $W_k\in C(\overline{\Omega})$ and
 \begin{equation*}
 	W_k\big|_{\partial\Omega}=\psi_k\big|_{\partial\Omega}=\varphi.
 \end{equation*}
 Now, by   \eqref{case4-F-3} and Lemma \ref{lemma2}, the equation \eqref{equation}    admits a convex viscosity solution $u\in C(\overline{\Omega})$ satisfying $u\big|_{\partial\Omega}= \varphi$, i.e.  the problem \eqref{equation}-\eqref{bcondition} admits a viscosity solution $u\in C(\overline{\Omega})$.  Using \eqref{case4-F-3}  again, we can deduce from the Comparison Principle (Lemma \ref{lemma1}) that  
 \begin{equation*} 
 	W_k(x)\leq  u(x), \ \ \ \  \forall x\in\overline{\Omega}.
 \end{equation*}
Hence we can infer from \eqref{eq:W-case4-new} and Lemma \ref{lemma:dx}  that 
 \begin{align*}
 	u(x)
 	&\geq -C_0\left(-\phi_0(x)\right)^{\theta_0}+\psi_k(x)\\
 	&\geq -C_0\left\|\phi_0\right\|_{C^2(\overline{\Omega})}^{\theta_0}\cdot \left(d(x)\right)^{\theta_0}+\psi_k(x), \ \ \ \  \forall x\in\overline{\Omega}.
 \end{align*}
 Denote 
 $\overline{C_0}=C_0\left\|\phi_0\right\|_{C^2(\overline{\Omega})}^{\theta_0} $. Thus, by letting $k\to\infty$, we get
 \begin{equation*}
 	u(x)\geq \psi(x)-\overline{C_0}\left(d(x)\right)^{\theta_0}, \ \ \forall x\in\overline{\Omega},
 \end{equation*}
where $\theta_0=\frac{a+2b+\alpha}{a+b}$. 
This yields Theorem \ref{thm1} (iii).

Thus we have completed the proof of Theorem \ref{thm1}.

\section{Proof of Nonexistence Theorem}\label{Sec:Case 5}

Let us prove Theorem \ref{thm2} by contradiction. We suppose that  
 the problem \eqref{equation}-\eqref{bcondition} admits a viscosity solution $u$. By Definition \ref{def:vis}, we have $u\in C(\overline{\Omega})$. 

Consider
\begin{equation}\label{eq:W-case5}
	W=-C\left(-\phi_0\right)^{\theta}+\max_{\partial\Omega}\varphi,
\end{equation}
where $C> 0$  and $\theta>0$ are constants to be determined. We point out that the term $\max_{\partial\Omega}\varphi$ is indeed a constant. 
Let us turn to prove that for any $C>0$, there exists $\theta=\theta(C)$ sufficiently small such that 
$W$ is a classically strict super-solution to  \eqref{equation}.

It is straightforward to see that 
\begin{align*} 
	DW&=C\theta\left(-\phi_0\right)^{\theta-1}D\phi_0,\\
	D^2	W&=C\theta\left(-\phi_0\right)^{\theta-1}D^2\phi_0+C\theta(1-\theta)\left(-\phi_0\right)^{\theta-2}D\phi_0\otimes D\phi_0.
\end{align*}
Now we can apply Lemma \ref{lemma:eigenvalue},  Remark \ref{remark:rho}, Remark \ref{remark:product} and \eqref{phi0-max} to obtain
\begin{align*}
	\lambda_{\min}(D^2W)&=\min_{|\bm{\xi}|=1}\,\bm{\xi}^TD^2W\,\bm{\xi}=\min_{|\bm{\xi}|=1}\,\bm{\xi}^T\left(C\theta\left(-\phi_0\right)^{\theta-1}D^2\phi_0+C\theta(1-\theta)\left(-\phi_0\right)^{\theta-2}D\phi_0\otimes D\phi_0 \right)\bm{\xi}\\
	&\leq \bm{\xi_1}^T\left(C\theta\left(-\phi_0\right)^{\theta-1}D^2\phi_0+C\theta(1-\theta)\left(-\phi_0\right)^{\theta-2}D\phi_0\otimes D\phi_0 \right)\bm{\xi_1}\\
	&= C\theta\left(-\phi_0\right)^{\theta-1}\bm{\xi_1}^TD^2\phi_0\,\bm{\xi_1}+C\theta(1-\theta)\left(-\phi_0\right)^{\theta-2}\bm{\xi_1}^T D\phi_0\otimes D\phi_0\,\bm{\xi_1}\\
	&= C\theta\left(-\phi_0\right)^{\theta-1}\bm{\xi_1}^TD^2\phi_0\,\bm{\xi_1}\\
	&\leq C\theta\left(-\phi_0\right)^{\theta-1}\overline{\rho_0},\stepcounter{equation}\tag{\theequation}\label{lambda-min-case5}\\
	\lambda_{\max}(D^2W)&=\max_{|\bm{\xi}|=1}\,\bm{\xi}^TD^2W\,\bm{\xi}=\max_{|\bm{\xi}|=1}\,\bm{\xi}^T\left(C\theta\left(-\phi_0\right)^{\theta-1}D^2\phi_0+C\theta(1-\theta)\left(-\phi_0\right)^{\theta-2}D\phi_0\otimes D\phi_0 \right)\bm{\xi}\\
	&\leq C\theta\left(-\phi_0\right)^{\theta-1}\max_{|\bm{\xi}|=1}\,\bm{\xi}^TD^2\phi_0\,\bm{\xi}+C\theta(1-\theta)\left(-\phi_0\right)^{\theta-2}\max_{|\bm{\xi}|=1}\,\bm{\xi}^T D\phi_0\otimes D\phi_0\,\bm{\xi}\\
	&\leq C\theta\left(-\phi_0\right)^{\theta-1}\overline{\rho_0}+C\theta(1-\theta)\left(-\phi_0\right)^{\theta-2}|D\phi_0|^2\\
	&=C\theta \left[\left(-\phi_0\right)\overline{\rho_0}+(1-\theta)|D\phi_0|^2\right]\left(-\phi_0\right)^{\theta-2}\\
	&\leq C\theta \left[\eta_0\overline{\rho_0}+(1-\theta)\|\phi_0\|^2_{C^2(\overline{\Omega})}\right]\cdot\left(-\phi_0\right)^{\theta-2}.\stepcounter{equation}\tag{\theequation}\label{lambda-max-case5}
\end{align*}

In view of Lemma \ref{lemma:dx}, we also notice that 
\begin{equation}\label{dx-case5}
	 d(x)\leq \frac{\operatorname{diam}(\Omega)}{\eta_0}\cdot\left(-\phi_0(x)\right),\ \ \forall x\in\Omega.
\end{equation}
Then using the assumptions  $(\mathbf{A_2})$-$(\mathrm{ii})$ and $(\mathbf{A_3})$ together with  \eqref{lambda-min-case5}, \eqref{lambda-max-case5} and \eqref{dx-case5}, we obtain 
\begin{align*}
	&\ \ \ \ F\left(\lambda_1(D^2 W_k),\cdots,\lambda_n(D^2 W_k)\right)\cdot\left[g(x)(d(x))^{\alpha}\right]^{-1}\\
	&\leq \Lambda_0\left(\lambda_{\min}(D^2 W_k)\right)^{a}\left(\lambda_{\max}(D^2 W_k)\right)^{b}\cdot m_0^{-1}\left(d(x)\right)^{-\alpha}\\
	&\leq \Lambda_0  \left(C\theta\left(-\phi_0\right)^{\theta-1}\overline{\rho_0}\right)^a\left\{C\theta \left[\eta_0\overline{\rho_0}+(1-\theta)\|\phi_0\|^2_{C^2(\overline{\Omega})}\right]\cdot\left(-\phi_0\right)^{\theta-2}
	\right\}^bm_0^{-1}\left(\frac{\operatorname{diam}(\Omega)}{\eta_0}\right)^{-\alpha}\cdot\left(-\phi_0(x)\right)^{-\alpha}\\
	&=	 C^{a+b}\theta^{a+b} \Lambda_0 m_0^{-1}(\operatorname{diam}(\Omega))^{-\alpha}\eta_0{}^\alpha \overline{\rho_0}^a\left[\eta_0\overline{\rho_0}+(1-\theta)\|\phi_0\|^2_{C^2(\overline{\Omega})}\right]^b  
\cdot\left(-\phi_0\right)^{(a+b)\theta-(a+2b+\alpha)},\ \ \forall x\in\Omega.	
\end{align*}

We note  that 
\begin{equation*}
\lim_{\theta\to 0^+}\left[\eta_0\overline{\rho_0}+(1-\theta)\|\phi_0\|^2_{C^2(\overline{\Omega})}\right]^b=\left[\eta_0\overline{\rho_0}+\|\phi_0\|^2_{C^2(\overline{\Omega})}\right]^b.
\end{equation*}
Since $\alpha\leq -a-2b$, we have $-(a+2b+\alpha)\geq 0$. Then it follows from \eqref{phi0-max} that 
\begin{equation*}
\lim_{\theta\to 0^+}\left(-\phi_0\right)^{(a+b)\theta-(a+2b+\alpha)}=\left(-\phi_0\right)^{-(a+2b+\alpha)}\leq \eta_0{}^{-(a+2b+\alpha)}.
\end{equation*}
By virtue of the fact that 
 \begin{equation*}
 	\lim_{\theta\to 0^+}\theta^{a+b} =0,
 \end{equation*}
we are now ready to derive that for any $C>0$, there exists $\theta_0=\theta_0(C)>0$ sufficiently small (without loss of generality, we can require $\theta_0\in(0,1]$) such that for any $\theta\in(0,\theta_0]$, there holds
\begin{equation*}
F\left(\lambda_1(D^2 W),\cdots,\lambda_n(D^2 W)\right)\cdot\left[g(x)(d(x))^{\alpha}\right]^{-1}<1,\ \ \forall x\in\Omega,
\end{equation*}
that is, 
\begin{equation}\label{case5-F}
	F\left(\lambda_1(D^2 W),\cdots,\lambda_n(D^2 W)\right) <g(x)(d(x))^{\alpha},\ \ \forall x\in\Omega,
\end{equation}
Now we take $\theta=\theta_0$. Then we obtain that 
\begin{equation}\label{eq:W-case5-new}
	W=-C\left(-\phi_0\right)^{\theta_0}+\max_{\partial\Omega}\varphi
\end{equation} is a classically strict super-solution to  \eqref{equation} over $\Omega$.

According to the definition \eqref{eq:W-case5-new} and Lemma \ref{lemma:uniform-convex}, we have  $W\in C(\overline{\Omega})$ and
\begin{equation*}
	W\big|_{\partial\Omega}=\max_{\partial\Omega}\varphi\geq\varphi.
\end{equation*}
Therefore, by \eqref{case5-F} and the Comparison Principle (Lemma \ref{lemma1}), we can infer that 
\begin{equation*} 
	W(x)\geq  u(x), \ \ \ \  \forall x\in\overline{\Omega}.
\end{equation*}
In particular, it follows from  \eqref{eq:W-case5-new}, \eqref{phi0-max} and  \eqref{phi-0-2} that 
\begin{align*}
	u(x_0)
	&\leq W(x_0)= -C\left(-\phi_0(x_0)\right)^{\theta_0}+\max_{\partial\Omega}\varphi=-C\eta_0{}^{\theta_0}+\max_{\partial\Omega}\varphi.
\end{align*}
Since $\theta_0\in(0,1]$, we notice that 
\begin{equation*}
\eta_0{}^{\theta_0}\geq \min\{\eta_0{}^1,\eta_0{}^0\}=\min\{\eta_0,1\}>0.
\end{equation*}
Thus for any $C>0$, we can derive that 
\begin{align*}
	u(x_0)
	&\leq  -C\min\{\eta_0,1\}+\max_{\partial\Omega}\varphi.
\end{align*}
By letting $C\to+\infty$, we further obtain
\begin{equation*}
	u(x_0)\leq -\left(\lim_{C\to+\infty}C\right)\min\{\eta_0,1\}+\max_{\partial\Omega}\varphi=-\infty ,  
\end{equation*}
which contradicts the fact that $u\in C(\overline{\Omega})$. 
The proof of Theorem \ref{thm2} is now complete.

\section{Further discussions}\label{sec:discuss}
In the last section, we mainly discuss the intuitions behind why $-a-2b$ is a critical value of $\alpha$ that distinguishes between the existence and nonexistence of viscosity solutions to the problem \eqref{equation'}-\eqref{bcondition'}. 
In fact,  when $\alpha\in(-\infty,-a-2b]$, one can observe that the problem \eqref{equation'}-\eqref{bcondition'} never admits a viscosity solution $u$  
 even if we assume the smoothness of $\Omega$ and $\varphi$, i.e. assume that $\Omega$ is of $C^\infty$ and $\varphi\in C^\infty(\partial\Omega)$, which together with Lemma \ref{lemma:approximation} further yields $\psi_k\in   C^\infty(\overline{\Omega})$. 
 
Now we let $\alpha\in(-\infty,-a-2b]$. In what follows, we use the singularity order to describe the speed that approaches the singularity (i.e. tends to infinity). We are going to analyze the specific singularity order of $\lambda_{\min}(D^2u)$ and $\lambda_{\max}(D^2u)$ by using the method of undetermined coefficients. Intuitively, from \eqref{equation'}, we  assume the following relation for the norm of the gradient $Du$ near the boundary:
\begin{equation}\label{relation1}
|Du|\sim  \big(d(x)\big)^{-\widetilde{\theta}},\ \ \text{as $x\to\partial\Omega$}.
\end{equation}
By choosing some suitable coordinate system, we can assume that 
\begin{equation}\label{relation2}
	Du\sim\Big(\underbrace{0,\cdots,0}_{n-1},-\big(d(x)\big)^{-\widetilde{\theta}}\Big),
\end{equation}
which also implies
\begin{equation}\label{relation7}
	 \partial_{x_n}u\sim -\big(d(x)\big)^{-\widetilde{\theta}},\ \ \text{as $x\to\partial\Omega$}.
\end{equation}

By the analysis in \eqref{eq:eigen-1}-\eqref{eq:eigen-2}, 
we note that near the boundary, the second tangential derivative has the same singularity order as $\lambda_{\min}(D^2u)$, and the second normal derivative has the same singularity order as $\lambda_{\max}(D^2u)$. Now we have the following singularity order relations when the boundary is approached:
\begin{align*}
&D^2_{x'x'}u\sim\lambda_{\min}(D^2u),\ \ \text{as $x\to\partial\Omega$},\stepcounter{equation}\tag{\theequation}\label{relation3}\\
&\partial_{x_nx_n}u\sim \lambda_{\max}(D^2u),\ \ \text{as $x\to\partial\Omega$}.\stepcounter{equation}\tag{\theequation}\label{relation4}
\end{align*}

Fix a point $z_0\in\partial\Omega$. Up to suitable translations and rotations, we can assume that $z_0$ is the origin $\mathbf{0}$ and  
\begin{equation*}
\Omega\cap B_\rho(\mathbf{0})=\left\{(x',x_n)=(x_1,\cdots,x_n):\, x_n\geq \widetilde{\varphi}(x')=\sum_{i=1}^{n-1}\frac{\kappa_i}{2}x_i^2+o(|x'|^2)\right\},
\end{equation*}
where $\rho>0$ is a constant, and  $\widetilde{\varphi}:\mathbb{R}^{n-1}\to\mathbb{R}$ is a uniformly convex $C^2$ function such that
\begin{equation*}
\widetilde{\varphi}(0)=0,\ \ \ \ \nabla_{x'}\widetilde{\varphi}(0)=0,\ \ \ \ \text{and}\ \ \ \ D^2_{x'x'}\widetilde{\varphi}(0)=\operatorname{diag}\big(\kappa_1,\cdots,\kappa_{n-1}\big).
\end{equation*} 
Denote $$w=u-\psi_k.$$ By  \eqref{bcondition} and Lemma \ref{lemma:approximation}, we have $$w\big|_{\partial\Omega}=0.$$ 
 It follows that 
\begin{equation*}
D^2_{x'x'}w(\mathbf{0})=-\partial_{x_n}w(\mathbf{0})\operatorname{diag}\big(\kappa_1,\cdots,\kappa_{n-1}\big),
\end{equation*}
i.e.
\begin{equation*}
	D^2_{x'x'}u(\mathbf{0})-	D^2_{x'x'}\psi_k(\mathbf{0})=-\partial_{x_n}u(\mathbf{0})\operatorname{diag}\big(\kappa_1,\cdots,\kappa_{n-1}\big)+\partial_{x_n}\psi_k(\mathbf{0})\operatorname{diag}\big(\kappa_1,\cdots,\kappa_{n-1}\big),
\end{equation*}
that is,
\begin{equation*}
	D^2_{x'x'}u(\mathbf{0})=-\partial_{x_n}u(\mathbf{0})\operatorname{diag}\big(\kappa_1,\cdots,\kappa_{n-1}\big)+D^2_{x'x'}\psi_k(\mathbf{0})+\partial_{x_n}\psi_k(\mathbf{0})\operatorname{diag}\big(\kappa_1,\cdots,\kappa_{n-1}\big).
\end{equation*}
Due to the smoothness of $\psi_k$ over $\overline{\Omega}$, we note that both $D^2_{x'x'}\psi_k$ and $\partial_{x_n}\psi_k$ are bounded on the boundary. Hence the singularity order of $D^2_{x'x'}u$ on the boundary is primarily influenced  by  the singularity order of $\partial_{x_n}u$. Consequently, the second tangential derivative and the first normal derivative have the same singularity order near the boundary:
\begin{equation}\label{relation5}
	D^2_{x'x'}u\sim -\partial_{x_n}u,\ \ \text{as $x\to\partial\Omega$}.
\end{equation}

We recall that the gradient of the distance function $d(x)$ is precisely $N$, the unit inward normal to $\partial\Omega$. This implies that when $x$ is near the boundary point $\mathbf{0}\in\partial\Omega$, we have  $$\partial_{x_n} d(x)\sim|D d(x)| =1,\ \ \text{as $x\to\mathbf{0}\in\partial\Omega$}.$$ 
Using the chain rule leads us to 
\[\partial_{x_nx_n}u(x)=\partial_{d(x)}\partial_{x_n}u(x)\cdot \partial_{x_n} d(x)\sim \partial_{d(x)}\partial_{x_n}u(x),\ \ \text{as $x\to\mathbf{0}\in\partial\Omega$}.\]
Thus  we have 
\[\partial_{x_nx_n}u(x)\sim \partial_{d(x)}\partial_{x_n}u(x),\ \ \text{as $x\to\mathbf{0}\in\partial\Omega$}.\]
Now by \eqref{relation7}, we obtain that the second normal derivative has a singularity order that is one order greater than that of the first normal derivative near the boundary, that is,
\begin{equation}\label{relation8}
\partial_{x_nx_n}u\sim \big(d(x)\big)^{-\widetilde{\theta}-1},\ \ \text{as $x\to\partial\Omega$}.
\end{equation}
 
It follows from \eqref{relation7}, \eqref{relation3}, \eqref{relation4}, \eqref{relation5} and \eqref{relation8} that
\begin{align*}
&\lambda_{\min}(D^2u)\sim  \big(d(x)\big)^{-\widetilde{\theta}},\ \ \text{as $x\to\partial\Omega$},\\
&\lambda_{\max}(D^2u)\sim   \big(d(x)\big)^{-\widetilde{\theta}-1},\ \ \text{as $x\to\partial\Omega$}.
\end{align*}
By virtue of \eqref{equation'}, i.e. 
\[\Lambda_0\left(\lambda_{\min}(D^2u)\right)^{a}\left(\lambda_{\max}(D^2u)\right)^{b}=g(x)(d(x))^{\alpha},\]
we can derive from the exponent of $d(x)$ on both sides to get 
\begin{equation*}
	a\cdot (-\widetilde{\theta})+b\cdot (-\widetilde{\theta}-1)=\alpha.
\end{equation*}
Considering $\alpha\in(-\infty,-a-2b]$ now, we infer that  
\begin{equation}\label{theta}
	\widetilde{\theta}\geq 1.
\end{equation}

Up to now, in the case $\alpha\in(-\infty,-a-2b]$, for any boundary point, we can choose proper coordinate system and summarize what we have obtained in \eqref{relation1}-\eqref{theta} as the following   singularity order relations:
\begin{align*}  
	&\lambda_{\min}(D^2u)\sim D^2_{x'x'}u\sim   -\partial_{x_n}u\sim|Du|\sim   \big(d(x)\big)^{-\widetilde{\theta}},\ \ \text{as $x\to\partial\Omega$},\\ 
	&\lambda_{\max}(D^2u)\sim \partial_{x_nx_n}u\sim   \big(d(x)\big)^{-\widetilde{\theta}-1},\ \ \text{as $x\to\partial\Omega$},
\end{align*}
where $\widetilde{\theta}\geq 1$. 
In particular, when $\alpha=-a-2b$, we have the following  relations of singularity order:
\begin{align*}  
	&\lambda_{\min}(D^2u)\sim D^2_{x'x'}u\sim   -\partial_{x_n}u\sim|Du|\sim \big(d(x)\big)^{-1},\ \ \text{as $x\to\partial\Omega$},\\ 
	&\lambda_{\max}(D^2u)\sim \partial_{x_nx_n}u\sim  \big(d(x)\big)^{-2},\ \ \text{as $x\to\partial\Omega$}.
\end{align*}

Suppose there exists a viscosity solution $u$. By Definition \ref{def:vis}, we have $u\in C(\overline{\Omega})$.  Let   $z(t)$ for $0\leq t\leq\varepsilon_0$ be a curve segment in $\overline{\Omega}$ starting from the point $z_0\in\partial\Omega$  that 
follows the direction of the fastest decrease of $u(x)$. Hence $z(t)$ satisfies the following conditions:
\begin{align*}
&z(0)=z_0,\\
&z(t)\in\overline{\Omega},\ \ \forall 0\leq t\leq \varepsilon_0,\\
&\frac{d}{dt}u(z(t))=-|Du(z(t))|,\ \ \forall 0\leq t\leq \varepsilon_0.
\end{align*}
Without loss of generality, we can assume that $t$ represents the arc length parameter of the curve $z(t)$. It is clear that the distance from the point $z(t)$ to the boundary is no greater than the arc length from the point  $z(t)$ to the point $z(0)$ along the curve $z(t)$, which means
\[d(z(t))\leq t.\]
Then due to \eqref{theta},  we can derive 
\begin{align*}
u(z(\varepsilon_0))-u(z_0)&=u(z(\varepsilon_0))-u(z(0))=\int_0^{\varepsilon_0}\frac{d}{dt}u(z(t))dt=-\int_0^{\varepsilon_0}|Du(z(t))|dt\\
&\sim -\int_0^{\varepsilon_0}\big(d(z(t))\big)^{-\widetilde{\theta}}dt\leq -\int_0^{\varepsilon_0}t^{-\widetilde{\theta}}dt=-\infty,
\end{align*}
which implies that 
\[u(z(\varepsilon_0))\leq u(z_0)-\infty 
=\varphi(z_0)-\infty=-\infty.\]
This contradicts the fact that $u\in C(\overline{\Omega})$. Therefore  the problem \eqref{equation'}-\eqref{bcondition'} does not admit any viscosity solutions  when $\alpha\in(-\infty,-a-2b]$.


\begin{thebibliography}{99}


\bibitem{Caffarelli-Nirenberg-Spruck-I} Caffarelli, L.; Nirenberg, L.; Spruck, J. The Dirichlet problem for nonlinear second-order elliptic equations. I. Monge-Ampère equation.
Comm. Pure Appl. Math. 37 (1984), no. 3, 369-402.

\bibitem{Caffarelli-Nirenberg-Spruck-III} Caffarelli, L.; Nirenberg, L.; Spruck, J. The Dirichlet problem for nonlinear second-order elliptic equations.  III. Functions
of the eigenvalues of the Hessian. Acta Math. 155 (1985), no. 3-4, 261-301.


\bibitem{C}  Calabi, E.  
Improper affine hyperspheres of convex type and a generalization of a theorem by K. Jörgens.
Michigan Math. J. 5 (1958), 105-126.
	
	\bibitem{Cheng-Yau}
Cheng, Shiu Yuen; Yau, Shing Tung. 
On the regularity of the Monge-Amp\`ere equation $\det \frac{\partial^2u}{\partial x_i\partial x_j}=F(x,u)$.
Comm. Pure Appl. Math. 30 (1977), no. 1, 41-68. 
	
 
	
\bibitem{Crandall}	Crandall, Michael G. Viscosity solutions: a primer. Viscosity solutions and applications (Montecatini Terme, 1995), 1–43, Lecture Notes in Math., 1660, Fond. CIME/CIME Found. Subser., Springer, Berlin, 1997. 

\bibitem{D-Z} Delfour, Michel C.; Zolésio, Jean-Paul. 
Shape analysis via oriented distance functions. 
J. Funct. Anal. 123 (1994), no. 1, 129-201.

\bibitem{Figalli}  Figalli, Alessio. The Monge-Amp\`ere equation and its applications. Zurich Lectures in Advanced Mathematics. European Mathematical Society (EMS), Z\"urich, 2017. 

\bibitem{G-T} Gilbarg, David; Trudinger, Neil S. Elliptic partial differential equations of second order. Classics Math. Springer-Verlag, Berlin, 2001, 517 pp.

 

\bibitem{G-Huang} Gutiérrez, Cristian E.; Huang, Qingbo.  The refractor problem in reshaping light beams.  Arch. Ration. Mech. Anal. 193 (2009), no. 2, 423-443.

\bibitem{Guanbo} Guan, Bo. 	The Dirichlet problem for a class of fully nonlinear elliptic equations.	Comm. Partial Differential Equations 19 (1994), no. 3-4, 399–416.
 
\bibitem{Guan-Jian} Guan, Bo; Jian, Huai-Yu. The Monge-Ampère equation with infinite boundary value. 
Pacific J. Math. 216 (2004), no. 1, 77-94.

 




\bibitem{Ishii-1} Ishii, Hitoshi. Perron’s method for Hamilton-Jacobi equations, Duke Math. J. 55 (2) (1987), 369-384.

\bibitem{Ishii-2} Ishii, Hitoshi. 
On uniqueness and existence of viscosity solutions of fully nonlinear second-order elliptic PDEs. 
Comm. Pure Appl. Math. 42 (1989), no. 1, 15-45.

\bibitem{Jiang-Trudinger-Yang} Jiang, Feida; Trudinger, Neil S.; Yang, Xiao-Ping. On the Dirichlet problem for Monge-Ampère type equations. Calc. Var. Partial Differential Equations 49 (2014), no. 3-4, 1223-1236.

\bibitem{Jian} Jian, Huaiyu. Hessian equations with infinite Dirichlet boundary value. Indiana Univ. Math. J. 55 (2006), no. 3, 1045-1062.

\bibitem{Jian-Wang1} Jian, Huaiyu; Wang, Xu-Jia. Existence of entire solutions to the Monge-Ampère equation. Amer. J. Math. 136 (2014), no. 4, 1093-1106.

 

\bibitem{J} Jörgens, K. 
Über die Lösungen der Differentialgleichung $rt-s^2=1$. (In German).
Math. Ann. 127 (1954), 130–134.

 



\bibitem{Liyanyan} Li, Yan Yan. Some existence results for fully nonlinear elliptic equations of Monge-Ampère type. 
Comm. Pure Appl. Math. 43 (1990), no. 2, 233–271.

\bibitem{Li-Li-2022} Li, You; Li, Mengni. 
Boundary Hölder regularity for a class of fully nonlinear elliptic partial differential equations. 
Nonlinear Anal. 216 (2022), Paper No. 112681, 19 pp.

\bibitem{Li-Li-2025} Li, Mengni; Li, You. 
Existence, uniqueness and interior regularity of viscosity solutions for a class of Monge-Ampère type equations. 
J. Differential Equations 415 (2025), 202-234.

 

 

\bibitem{P} Pogorelov, A. V. The Minkowski multidimensional problem. Translated from the Russian by Vladimir Oliker. Introduction by Louis Nirenberg. Scripta Series in Mathematics. V. H. Winston \& Sons, Washington, D.C.; Halsted Press [John Wiley \& Sons], New York-Toronto-London, 1978. 

\bibitem{Yuanyu}Shankar, Ravi; Yuan, Yu. Hessian estimates for the sigma-2 equation in dimension four. 
Ann. of Math. (2) 201 (2025), no. 2, 489-513.

 

 


\bibitem{Urbas90} Urbas, John I. E. On the existence of nonclassical solutions for two classes of fully nonlinear elliptic equations. 
Indiana Univ. Math. J. 39 (1990), no. 2, 355-382.
	
	
 

\bibitem{Viaclovsky} Viaclovsky, Jeffrey Alan. 
Conformal geometry, contact geometry, and the calculus of variations. 
ProQuest LLC, Ann Arbor, MI, 1999, 67 pp.

\bibitem{Wang-Jiang}
Wang, Jiangwen; Jiang, Feida. Regularity of solutions for degenerate or singular fully nonlinear integro-differential equations. Commun. Contemp. Math.  (2025), Online Ready.
	
\bibitem{Wang-2009} Wang, Xu-Jia. 
The  $k$-Hessian equation. Lecture Notes in Mathematics (LNMCIME,volume 1977), 
Springer, Dordrecht, 2009, 177-252.

\bibitem{Yuan} Yuan, Yu. Special Lagrangian equations. 
Progr. Math., 333. 
Birkhäuser/Springer, Cham, 2020, 521-536.
\end{thebibliography}
\end{document}